\documentclass[a4paper]{article}

\usepackage[T1]{fontenc}
\usepackage[utf8]{inputenc}
\usepackage[english]{babel}
\usepackage{microtype}
\usepackage[breaklinks]{hyperref}
\usepackage{amsmath}
\usepackage{amsfonts}
\usepackage{amsthm}
\usepackage{amssymb}
\usepackage[capitalize]{cleveref}
\usepackage{subcaption}
\usepackage{url}
\usepackage{pgfplots}
\usepackage{tikz}
\usetikzlibrary{arrows}
\usetikzlibrary{decorations.markings}

\crefname{section}{Section}{Sections}
\crefname{subsection}{Section}{Sections}
\crefname{equation}{}{}

\theoremstyle{plain}
\newtheorem{theorem}{Theorem}[section]
\newtheorem{lemma}[theorem]{Lemma}

\newtheorem{definition}[theorem]{Definition}
\newtheorem{remark}[theorem]{Remark}

\pgfplotsset{compat=1.4,
	     width=5.9cm,
	     height=5.9cm
	    }

\newcommand{\ubar}[1]{\underline{#1\mkern-2mu}\mkern2mu }
\newcommand{\Snap}{\mathcal{S}}
\newcommand{\AccSnap}{\widetilde{\mathcal{S}}}
\newcommand{\ProjSnap}{\widetilde{\mathcal{R}}}
\newcommand{\SnapMap}{\ubar{\mathcal{S}}}
\newcommand{\AccSnapMap}{\widetilde{\ubar{\mathcal{S}}}}
\newcommand{\ProjSnapMap}{\widetilde{\ubar{\mathcal{R}}}}
\newcommand{\POD}{\operatorname{POD}}
\newcommand{\MPOD}{\overline{\operatorname{POD}}}
\newcommand{\HAPOD}{\operatorname{HAPOD}}
\newcommand{\HAPODSTDE}[1][\T]{\HAPOD[\Snap, {#1}, D, \varepsilon]}
\newcommand{\Span}{\operatorname{span}}

\newcommand{\T}{\mathcal{T}}
\newcommand{\Tr}{\operatorname{tr}}
\newcommand{\children}[2][\T]{\mathcal{C}_{#1}(#2)}
\newcommand{\childmap}[1][\T]{\mathcal{C}_{#1}}

\newcommand{\R}{\mathbb{R}}

\title{Hierarchical Approximate Proper Orthogonal Decomposition%
\thanks{Supported by the Deutsche Forschungsgemeinschaft, DFG EXC 1003 Cells in Motion - Cluster of Excellence, M\"unster, Germany,
by the Center for Developing Mathematics in Interaction, DEMAIN, M\"unster, Germany,
by Cells in Motion (CiM) Cluster of Excellence in flexible funds project FF-2015-07,
<<<<<<< HEAD
by the German Federal Ministry of Education and Research (BMBF) under contract number 05M13PMA
and by the German Federal Ministry for Economic Affairs and
Energy (BMWi), in the joint project: ``MathEnergy -- Mathematical
Key Technologies for Evolving Energy Grids'', sub-project:
Model Order Reduction (Grant number: 0324019\textbf{B}).
}}

\author{
Christian Himpe\thanks{Contact: \href{mailto:himpe@mpi-magdeburg.mpg.de}{\nolinkurl{himpe@mpi-magdeburg.mpg.de}},  
Computational Methods in Systems and Control Theory Group at the Max Planck Institute for Dynamics of Complex Technical Systems, Sandtorstra{\ss}e~1, D-39106 Magdeburg, Germany}
\and
Tobias Leibner\thanks{Contact: \href{mailto:tobias.leibner@uni-muenster.de}{\nolinkurl{tobias.leibner@uni-muenster.de}},
Applied Mathematics, University of M\"unster, Einsteinstrasse~62, D-48149 M\"unster, Germany}
\and
Stephan Rave\thanks{Contact: \href{mailto:stephan.rave@uni-muenster.de}{\nolinkurl{stephan.rave@uni-muenster.de}},
Applied Mathematics, University of M\"unster, Einsteinstrasse~62, D-48149 M\"unster, Germany}
}

\date{June 21, 2017}

\begin{document}

\maketitle

\begin{abstract}
Proper Orthogonal Decomposition (POD) is a widely used technique for the construction of
low-dimensional approximation spaces from high-dimensional input data.
For large-scale applications and an increasing number of input data vectors, however,
computing the POD often becomes prohibitively expensive.
This work presents a general, easy-to-implement approach to compute an approximate POD
based on arbitrary tree hierarchies of worker nodes, where each worker computes a POD of
only a small number of input vectors.
The tree hierarchy can be freely adapted to optimally suit the available computational resources. 
In particular, this hierarchical approximate POD (HAPOD) allows for both simple
parallelization with low communication overhead, as well as incremental POD computation under
constrained memory capacities.
Rigorous error estimates ensure the reliability of our approach, and extensive numerical examples underline its
performance.
\end{abstract}

\section{Introduction}
The construction of low-dimensional subspaces from high-dimensional data, dynamics or operators is an essential mechanism in many applications,
with the aim to accelerate or merely enable numerical computations of large-scale models.
In the discipline of model reduction, this methodology is the central problem under investigation.

A well-known and popular approach for subspace construction is the Proper Orthogonal Decomposition (POD),
i.e.\ the computation of the left-singular vectors associated with the dominant singular values
of a given set of input column vectors concatenated to a matrix.
An important field of application for the POD is the reduction of ordinary differential equation (ODE) models \cite{moore79} and partial differential equation (PDE) models \cite{kunisch99,kunisch02}.
A landmark work in this context is the use of the POD for compression of simulation data \cite{sirovich87} where the dominant modes are extracted from flow simulation time series by the \emph{method of snapshots}.
For an elaborate review of the POD method see for example \cite{gubisch17,holmes12}.

Due to technical limitations of computational resources, such as memory-space and acceptable computational complexities, not only the evaluation of a large-scale problem,
but even the computation of a low-rank approximation by existing methods may be infeasible.
This is particularly true for the POD, as the (truncated) singular value decomposition (SVD) of large matrices
is a computationally demanding task.
In order to speed up the computation,
various parallel algorithms are available for SVD computation \cite{berry05};
more recently, partitioning approaches were developed to obtain the SVD, or an approximation thereof, such as \cite{solovyev14a,solovyev14},
\cite{constantine11,constantine14}, \cite{beattie06}, \cite{wang15}, as well as a related parallel QR decomposition in \cite{sayadi14}.
A commonality of these methods is the horizontal slicing of the argument matrix, which is similar to the partitioning of the spatial domain of a discretized PDE model.
However, such an approach is only possible when complete horizontal slices of the argument matrix are available.
This usually means that all input data vectors have to be computed and stored before starting the POD computation.
For large problems, this might be impossible due to insufficient memory or even mass storage space.
Also, for parametrized problems the input data might be distributed column-wise among several workers,
and horizontal slicing of the input would require heavy communication between the workers, which may
be impossible, for instance in grid-computing environments.

In comparison, the herein proposed Hierarchical Approximate Proper Orthogonal Decomposition (HAPOD) is based on a vertical slicing of the input matrix and is 
targeted to extend POD-based methods which were designed with ``tall and skinny'' matrices in mind towards settings
where, due to enhanced requirements such as parametrization, the actual matrix dimension is ``tall and not-so-skinny''.

Our method is based upon the simple idea of replacing subsets of input vectors by POD approximations of these, which
then form the input of additional POD steps.
As such, our algorithm can be applied on top of any pre-existing POD implementation. 
Being formulated for arbitrary tree hierarchies of workers, it allows sequential and parallel decompositions,
as well as combinations thereof, based on the partitioning of the time domain or parameter space.

The HAPOD is a \emph{single pass} method in the sense that the input vectors at a given HAPOD node are
only required for a single local POD computation and can be discarded afterwards.
Rigorous error estimates allow a priori control of the final \mbox{$\ell^2$-approximation}
error for the input data. At the same time, bounds for the number of generated HAPOD modes guarantee
quasi-optimality of the generated approximation space.
As long as the final depth of the HAPOD tree is known, local PODs can be computed as soon
as all input data for a given node is available.
As such, the HAPOD can also be seen as a
general methodology for approximation quality control when updating POD spaces with additional input data.

Stochastic methods for SVD computation, e.g. \cite{MonteCarloSVD,halko11,QUICSVD,RandomPCA}, share many benefits with
the HAPOD.
In particular, these methods are easily parallelizable with comparable communication requirements
(at least when no power iteration is performed), and single pass formulations do exist.
However, most algorithms are designed for a prescribed fixed approximation rank.
Those which do guarantee spaces with prescribed approximation error (\cite[Section 4.4]{halko11}, \cite{QUICSVD}
or the preprint \cite{R3SVD}) are based on iterative procedures which require multiple passes
over the input data.
Also, our approach can be implemented more easily on top of already existing POD codes.

For the incremental (updated) computation of an SVD we refer to the work of Brand \cite{brand02,brand06},
which allows the update of an existing SVD given new data.
Geared towards (POD-based) model reduction, \cite{oxberry17} uses Brand's algorithm for an incremental POD algorithm.
In this context, the HAPOD framework provides 
local choices of truncation error tolerances to rigorously control the overall approximation error,
given that the maximum number of updates is known. 
A similar updated POD algorithm is employed for the experiment in \cref{sc:boltzmann}.
In \cite{baker12}, another family of rank-based approaches for incremental SVD computations is presented.

Given the simplicity of the HAPOD, we do not claim to be first in investigating this concept.
In fact, we recently became aware of \cite{Paul-Dubois-TaineAmsallem2015}, wherein special cases of
the HAPOD (i.e.\ a distributed and an incremental HAPOD in the sense of \cref{sc:special_cases}) are briefly discussed.
Balanced n-ary tree structures are investigated in \cite{iwen17}.
In both cases, error bounds for prescribed truncation ranks are derived.
Another application of the distributed HAPOD is discussed in \cite{NestedPOD}, which uses the error bound derived
in \cite{Paul-Dubois-TaineAmsallem2015}.
In the context of principal component analysis, distributed methods have been introduced \cite{macau10,qi04,qu02},
which, apart from the centering of the data set, correspond to a distributed HAPOD.
No rigorous error bounds are derived, however.

Main contribution of this work is a thorough study of the HAPOD with the aim of showing that it should be a
standard part in the toolbox of every model reduction practitioner. 
In particular, in contrast to \cite{Paul-Dubois-TaineAmsallem2015,iwen17,NestedPOD}, we formally analyze the algorithm in
a general setting with arbitrary tree topologies, making it suitable to more complex applications (cf.\ \cref{sc:boltzmann}),
and give, for prescribed local POD truncation error tolerances, estimates for both the approximation error
as well as the obtained (local and final) numbers of POD modes.
Based on these estimates we provide rules for the selection of the local error tolerances to achieve a given
global target (mean) approximation error, with a user-definable tradeoff between optimality of the generated
approximation space and computational efficiency. 
We show the performance of our method for input data with quickly decaying singular
values, as it is typically the case in model reduction applications (cf.\ \cref{rem:hapod_efficiency,rem:hapod_bounded_by_nwidths,sc:applic}).
\cref{sc:numex} contains extensive numerical experiments highlighting the applicability of our method.

Before introducing the HAPOD in \cref{sc:hapod}, we start with a 
concise summary of the POD and its properties in \cref{sc:pod}.

\section{Proper Orthogonal Decomposition}\label{sc:pod}

Proper Orthogonal Decomposition is a technique for finding
low-order approximation spaces for a given set of snapshot (data)
vectors by computing the left-singular vectors corresponding with the dominant singular values of the matrix
formed by the column-wise concatenation of the snapshot vectors.
Designations used in other fields are \emph{Principal Component Analysis},
\emph{Empirical Eigenfunctions}, \emph{Empirical Orthogonal Functions} or
\emph{Karhunen-Lo\`eve Decomposition}.
A more formal definition of the POD, which also applies to infinite-dimensional spaces,
is given as follows: 

\begin{definition}[Proper Orthogonal Decomposition (POD)]
Let $\Snap$ be a finite multiset of vectors contained in a Hilbert space $V$
and denote by $|\Snap|$ its cardinality.
With $e_1, \ldots, e_{|\Snap|} \in \R^{|\Snap|}$ the canonical
basis of $\R^{|\Snap|}$, and $\{s_1, \ldots, s_{|\Snap|}\} = \Snap$ an arbitrary enumeration of the
elements of $\Snap$,
we call sequences $\varphi_1, \ldots, \varphi_{|\Snap|} \in V$, $\sigma_1, \ldots, \sigma_{|\Snap|} \in \R$
{Proper Orthogonal Decomposition} modes and singular values of
$\Snap$ if $\varphi_m$, $\sigma_m$ are the left-singular vectors and singular
values of the linear mapping
$\SnapMap$ given by
\begin{equation}\label{eq:snapshot_map}
\SnapMap: \R^{|\Snap|} \to V,\quad e_m \mapsto \SnapMap(e_m):=s_m
\qquad 1 \leq m \leq |\Snap|.
\end{equation}
\end{definition}

\begin{remark}
Due to the uniqueness properties of the SVD, the POD
singular values of a given multiset $\Snap$ are uniquely defined. 
The POD modes are uniquely defined up to orthogonal mappings of subspaces of $V$
spanned by modes with the same singular value.
\end{remark}

\begin{remark}\label{def:method_of_snapshots}
	A simple yet numerically robust algorithm for the computation of the SVD
	of $\SnapMap$ is based on computing the eigenvalue decomposition of the Gramian
	$G:= (s_i, s_j)_{i,j}$ to the snapshot set $\Snap = \{s_1, \ldots, s_{|\Snap|}\}$.
	The $k$-th POD mode $\varphi_k$ is then obtained as
    \begin{equation*}
       \varphi_k = \frac{1}{\sqrt{\lambda_k}} \sum_{i=1}^{|\Snap|} \psi_{k,i} \cdot s_i,
    \end{equation*}
    where $\lambda_k$ is the $k$-th largest eigenvalue of G and $\psi_{k,i}$ the $i$-th component
    of the corresponding eigenvector.\footnote{Note that the condition number of $G$ is the square of the condition number
    of $\SnapMap$, limiting the numerical accuracy of this method in comparison to other
    SVD algorithms.}
\end{remark}

The basic idea of the algorithm outlined in \cref{def:method_of_snapshots}, which in the context of model reduction is
also known as \emph{method of snapshots} \cite{sirovich87}, is to replace the difficult
task of computing the SVD of a large snapshot matrix with the easier task of computing
the eigenvalue decomposition of the much smaller (symmetric) Gramian, which can be obtained efficiently by optimized matrix-matrix multiplication algorithms.

While this approach performs well if there are relatively few snapshot vectors
(i.e. ``tall and skinny'' snapshot matrices), it suffers from the quadratic
growth in computational complexity for computing the Gramian when the number of snapshots increases.
However, using this method in conjunction with the herein proposed HAPOD algorithm can drastically reduce
the overall required computational effort, making it feasible even for large snapshot
sets $\Snap$ (see \cref{sc:applic}).

The main reason for the importance of the POD is the fact that it produces best approximating spaces
in the $\ell^2$-sense:

\begin{theorem}[Schmidt-Eckhard-Young-Mirsky] \label{thm:pod_error}
Let $(\sigma_m, \varphi_m)$, $1 \leq m \leq |\Snap|$ be the singular values and modes
of a POD of a given snapshot multiset $\Snap$.
Then for each $1 \leq N \leq |\Snap|$, $V_N:=\Span\{\varphi_1, \ldots, \varphi_N\}$
is an $\ell^2$-best approximating space for $\Snap$ in the sense that
\begin{equation}\label{eq:pod_error}
	\sum_{s \in \Snap} \|s - P_{V_N}(s)\|^2 = \min_{\substack{X \subseteq V\\ \dim X = N}}\ \sum_{s \in \Snap} \|s - P_{X}(s)\|^2 = \sum_{m=N+1}^{|\Snap|} \sigma_m^2,
\end{equation}
where $\|\cdot\|$ denotes the norm on $V$ and $P_{X}$ is the $V$-orthogonal projection onto the linear subspace $X$.
\end{theorem}

The HAPOD algorithm presented in \cref{sc:hapod} can be based on any pre-existing POD implementation.
We formalize the concept of a POD algorithm as follows:

\begin{definition}\label{def:pod}
For a given Hilbert space $V$, let $\POD$ be the mapping
\begin{equation*}
	(\Snap, \varepsilon) \mapsto \POD(\Snap, \varepsilon):= \{(\sigma_n, \varphi_n)\}_{n=1}^N,
\end{equation*}
which assigns to each finite multiset $\Snap\subseteq V$ and each $\varepsilon > 0$ the set given
by the first $N$ pairs of singular values $\sigma_n$ and modes $\phi_n$
of the POD of $\Snap$, where $N$ is the smallest nonnegative integer such that the $\ell^2$-best-approximation
error is bounded by $\varepsilon$, i.e.\ 
	$\sum_{s \in \Snap} \|s - P_{V_N}(s)\|^2 \leq \varepsilon^2$.
According to \cref{eq:pod_error}, $N$ is thus given as:
\begin{equation*}
	N = \min \biggl\{ N^\prime \in \{0,\ldots,|\Snap|\} \biggm| \sum_{n=N^\prime+1}^{|\Snap|} \sigma_n^2 \leq
	\varepsilon^2\biggr\}.
\end{equation*}
Assuming that no SVD is performed for $\varepsilon = 0$ and the original snapshot multiset is returned,
we also define $\POD(\Snap, 0):= \{(1, s)\,|\, s \in \Snap\}$.
\end{definition}

\section{Hierarchical Approximate POD (HAPOD)}\label{sc:hapod}

In this section we introduce the HAPOD algorithm (\cref{sc:hapoddef}) and provide estimates
that allow to control the approximation error as well as the number of computed POD modes (\cref{sc:main_theorems}).
Special cases for distributed and incremental HAPOD computation are discussed in \cref{sc:special_cases}.
A further discussion of the advantages of the HAPOD is contained in \cref{sc:applic}, whereas
proofs of our main theorems can be found in \cref{sc:proofs}.
The notation used in this section is summarized in \cref{tab:notation}.

\begin{table}[h]\centering
\begin{tabular}{l@{\hskip 2.0mm}l@{\hskip 10.0mm}l@{\hskip 2.0mm}l}
 $\childmap(\alpha)$ & children of node $\alpha$ in tree $\T$              &    $\mathcal{N}_\T$ & node set of tree $\T$\\
 $D$ & snapshot-to-leaf map                                   &    $\mathcal{N}_\T(\alpha)$ & nodes below $\alpha$ in tree $\T$\\
 $\varepsilon(\alpha)$ & error tolerance at node $\alpha$  &    $\rho_\T$ & root node of tree $\T$\\
 $L_\T$ & depth of tree $\T$                                  &    $\Snap$ & snapshot set\\
 $L_\T(\alpha)$ & level of node $\alpha$ in tree $\T$                      &    $\Snap_\alpha$ & input snapshots at node $\alpha$\\
 $\mathcal{L}_\T$ & leaf set of tree $\T$                                &    $\AccSnap_\alpha$ & snapshots below $\alpha$ in the tree
\end{tabular}
    \caption{Key notation. Additional notation required in the proofs of \cref{thm:hapod_bound,thm:mode_bound} is
    given in \cref{def:hapod_notation}.}
\label{tab:notation}
\end{table}

\subsection{Definition of the HAPOD}\label{sc:hapoddef}
The basic idea of the HAPOD algorithm is to replace the task of computing a POD of a given
large snapshot set $\Snap$ by several small PODs, which only depend on small subsets
of $\Snap$ and previously computed PODs. To formalize this procedure, we consider
rooted trees where each node of the tree is associated with a local POD. 

A rooted tree is a connected acyclic graph of which one node is designated as the root
of the tree. The following equivalent definition will better suit our needs:

\begin{definition}[Rooted Tree]
For an arbitrary set $X$, denote by $\operatorname{Pow}(X)$ its power set.
We then call a triple $\T = (\mathcal{N}_\T, \childmap, \rho_\T)$,
where $\mathcal{N}_\T$ is a finite set, $\rho_\T \in \mathcal{N}_\T$, and $\childmap:
\mathcal{N}_\T \to \operatorname{Pow}(\mathcal{N}_\T \setminus \{\rho_\T\})$,
 a {rooted tree} if the mapping
$\childmap$ satisfies the following
properties:
\begin{gather}
	\forall \alpha, \beta \in \mathcal{N_\T}: \alpha \neq \beta \Rightarrow \children{\alpha} \cap \children{\beta} = \emptyset,
	\label{eq:treeonlyoneparent} \\ 
	\forall\, \emptyset \neq X \subseteq \mathcal{N}_\T \setminus \{\rho_\T\}\ \, \exists \alpha\in \mathcal{N}_\T
    \setminus X:\ \childmap(\alpha) \cap X \neq
	\emptyset.
	\label{eq:treeconnected}
\end{gather}
We call elements $\alpha \in \mathcal{N}_\T$ the {nodes} of $\T$ and the elements
of $\children{\alpha}$ the {children} of $\alpha$.
Condition \cref{eq:treeonlyoneparent} states that every node of
$\T$ is the child of at most one node, whereas condition
\cref{eq:treeconnected} ensures that every node is connected to
the {root node} $\rho_\T$. Together, \cref{eq:treeonlyoneparent,eq:treeconnected} imply that there are no cycles in $\T$.

The leaf set $\mathcal{L}_\T$ of $\T$ is given by
\begin{equation*}
	\mathcal{L}_\T:= \{\alpha \in \mathcal{N}_\T \ |\ \children{\alpha} = \emptyset \}.
\end{equation*}
For each node $\alpha \in \mathcal{N}_\T$ we define the {nodes below} $\alpha$,
$\mathcal{N}_\T(\alpha)$, recursively by the relation
\begin{equation*}
	\mathcal{N}_\T(\alpha) := \{\alpha\}\ \cup \bigcup_{\beta \in \children{\alpha}}
	\mathcal{N}_\T(\beta).
\end{equation*}
Finally, we define the {level} map $L_\T: \mathcal{N}_\T \to \mathbb{N}$ recursively as
\begin{equation*}
	L_\T(\alpha):= \max (\{L_\T(\beta) \ | \ \beta \in \children{\alpha}\} \cup \{0\}) + 1,
\end{equation*}
and call $L_\T := L_\T(\rho_\T)$ the {depth} of $\T$.
\end{definition}

Given a tree $\T$, the HAPOD algorithm works by first assigning vectors of a given snapshot
set $\Snap$ to the leaves of the tree. Then, starting with the leaves, a
POD of the local input data is computed at each node. The resulting modes are scaled by their corresponding
singular values and passed on as input to the parent node. The final HAPOD modes are collected
as the output of the root node $\rho_\T$ (cf.\ \cref{fig:special_cases,fig:boltzmannhapod}).
The precise definition is given as follows:

\begin{definition}[Hierarchical Approximate POD (HAPOD)]\label{def:hapod}
Let $\Snap \subseteq V$ be a finite multiset of snapshot vectors in a Hilbert space $V$.
Given a rooted tree $\T$ and mappings
\begin{equation*}
	D: \Snap \to \mathcal{L}_\T, \qquad \varepsilon: \mathcal{N}_\T \to \R^{\geq 0},
\end{equation*}
define recursively for each $\alpha \in \mathcal{N}_\T$
\begin{equation*}
\HAPODSTDE(\alpha):=\POD(\Snap_\alpha, \varepsilon(\alpha)),
\end{equation*}
where the local input data multiset $\Snap_\alpha$ is given by
\begin{equation*}
	\Snap_\alpha := 
	\begin{cases}
		D^{-1}(\{\alpha\}) & \alpha \in \mathcal{L}_\T,\\
		\bigcup_{\beta \in \children{\alpha}} \Bigl\{\sigma_n\cdot\varphi_n\
		|\ (\sigma_n, \varphi_n) \in \HAPODSTDE(\beta)\Bigr\} & \text{otherwise},
	\end{cases}
\end{equation*}
with $D^{-1}(\{\alpha\}) := \{s \in \Snap \,|\, D(s) \in \{\alpha\}\} = \{s \in \Snap \,|\, D(s) = \alpha\}$ being the multiset of all snapshot vectors
assigned to the leaf node $\alpha$.
We call $\HAPODSTDE:= \HAPODSTDE(\rho_\T)$ the {hierarchical approximate POD} of
$\Snap$ for the tree $\T$, the snapshot distribution $D$ and the local
tolerances $\varepsilon$.
\end{definition}

\subsection{Special Cases: Distributed and Incremental HAPOD}\label{sc:special_cases}
The HAPOD is defined for arbitrary rooted trees, yet two classes of tree topologies present important special cases
due to their ease of application. Both cases have also been discussed in \cite{Paul-Dubois-TaineAmsallem2015}.

One special case of the HAPOD constitutes a ``flat'' tree (star), in which all leaf nodes are the children of the root node,
i.e.\ $\children{\rho_{\T}} = \mathcal{N}_{\T} \setminus \{\rho_{\T}\}$,
and the snapshot set $\Snap$ is distributed evenly among the leaf nodes
(see \cref{fig:dapod}).
For such a tree the HAPOD is given as: 
\begin{multline*}
\HAPODSTDE[\T](\rho_{\T})= \\
\POD\Bigl(\bigcup_{\beta \in \mathcal{L}_{\T}} \Bigl\{\sigma_n\cdot\varphi_n\,\Bigm|\,(\sigma_n, \varphi_n) \in
\POD(D^{-1}(\{\beta\}),\varepsilon(\beta))\Bigr\},\, \varepsilon(\rho_{\T})\Bigr).
\end{multline*}

\begin{figure}
	\begin{subfigure}[t]{0.48\textwidth}\centering
		\raisebox{0.\height}{\begin{tikzpicture}
			[level distance=25mm,
				thin,
				every node/.style={draw, circle, text width=0.37cm, align=center},
				level 1/.style={sibling distance=15mm},
				edge from parent/.style={draw,>-,>=triangle 45 reversed}
			]
			\node {$\rho$}
				child {node {$\beta_1$}}
				child {node {$\beta_2$}}
				child {node {$\beta_3$}}
				child {node {$\beta_4$}};
			\end{tikzpicture}}\caption{Distributed approximate POD. The PODs at the leaves $\beta_i$ can be
			 			   computed in parallel. Afterwards an additional POD is performed at
						   the root node $\rho$.}
		\label{fig:dapod}
	\end{subfigure}
	\hfill
	\begin{subfigure}[t]{0.48\textwidth}\centering
		\begin{tikzpicture}
			[level distance=12mm,
				thin,
				every node/.style={draw, circle, text width=0.37cm, align=center},
				edge from parent/.style={draw,>-,>=triangle 45 reversed},
				sibling distance=20mm,
			]
			\node {$\rho$}
				child {node {$\alpha_3$}
					child {node {$\alpha_2$}
						child {node {$\alpha_1$}}
						child {node {$\beta_1$}}
					}
				child {node {$\beta_2$}}}
				child {node {$\beta_3$}};
			\end{tikzpicture}
			\caption{Incremental HAPOD. New snapshot data enters at the nodes $\beta_i$ which
				 is then combined with the current modes by PODs at the nodes $\alpha_i$.}
			\label{fig:rapod}
	\end{subfigure}
	\caption{Trees corresponding to distributed and incremental HAPOD computation.}\label{fig:special_cases}
\end{figure}
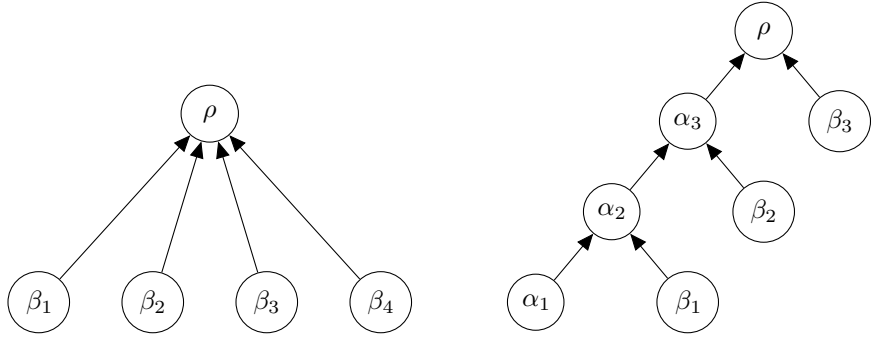

From a numerical linear algebra perspective this \emph{distributed HAPOD} is closely related to the ``horizontal slicing''
distributed SVD methods \cite{beattie06,constantine11,constantine14,sayadi14,solovyev14a,solovyev14,wang15}.
The key algorithmic difference is the horizontal partitioning of the data vectors forming the columns of the snapshot matrix
into fat chunks as opposed to the vertical partitioning into thin chunks of complete data vectors considered here. 

A second special case of the HAPOD is a ``skinny'' tree (totally unbalanced binary tree).
Each node of this tree is either a leaf or has exactly one leaf and one non-leaf as children (see \cref{fig:rapod}).
Formally, we then have $\mathcal{N}_{\T} = (\{\alpha_1, \ldots, \alpha_L\} \cup \{\beta_1, \ldots, \beta_{L-1}\})$, $\rho_{\T} = \alpha_L$,
$\children{\beta_l} = \emptyset$ for all $1 \leq l \leq L-1$,
$\children{\alpha_1} = \emptyset$ and
$\children{\alpha_l} = \{\alpha_{l-1}, \beta_{l-1}\}$ for $2 \leq l \leq L$.
Typically, one will perform no additional PODs on the input data, so $\varepsilon(\beta_l) = 0$.
In this case, the HAPOD is given as
    $\HAPODSTDE(\alpha_1)= \POD(D^{-1}(\{\alpha_1\}), \varepsilon(\alpha_1))$
and
\begin{multline*}
\HAPODSTDE(\alpha_l)=\\
\POD\Bigl(\bigl\{\sigma_n\cdot\varphi_n \ |\ (\sigma_n, \varphi_n) \in \HAPODSTDE(\alpha_{l-1})\bigr\} \cup D^{-1}(\{\beta_{l-1}\}),\, \varepsilon(\alpha_l)\Bigr).
\end{multline*}
for $2 \leq l \leq L$.
Thus, the HAPOD can be computed incrementally by a simple iterative procedure, where in each update step
a POD of the current (scaled) HAPOD modes together with the new input data is computed, whereas
old input data can be removed from memory.

To accelerate the computation of this \emph{incremental HAPOD}, an incremental SVD algorithm such as \cite{brand02}
might be used for the local POD computations.
In this case, then main theorems in \cref{sc:main_theorems} provide a means
to select truncation error tolerances for the individual SVD updates that guarantee
final approximation spaces of prescribed quality.

\subsection{Main Theorems}\label{sc:main_theorems}
Two central questions about the HAPOD are answered by the following theorems:
Given error tolerances $\varepsilon$, what is the approximation error for the computed HAPOD modes (\cref{thm:hapod_bound})?
How many modes does the HAPOD produce in comparison to a direct POD computation (\cref{thm:mode_bound})?
Only by controlling both quantities simultaneously can we arrive at an efficient approximation scheme.
The proofs to the following theorems are given in \cref{sc:proofs}.

\begin{theorem}\label{thm:hapod_bound}
Let $\Snap, \T, D, \varepsilon$ be given as in \cref{def:hapod},
let the multiset of all snapshots subordinate to the node $\alpha$ be given by $\AccSnap_\alpha := \bigcup_{\gamma
\in \mathcal{L}_\T \cap \mathcal{N}_\T(\alpha)} D^{-1}(\{\gamma\})$,
and let $P_{\alpha}$ be the $V$-orthogonal projection onto the linear space
spanned by the modes of $\HAPODSTDE(\alpha)$.
The $\ell^2$-approximation error for the HAPOD space at node~$\alpha$ is then bounded by:

\begin{equation}\label{eq:hapod-bound}
	\sum_{s \in \AccSnap_\alpha} \|s - P_\alpha(s)\|^2 \leq \ 
	\sum_{\gamma \in \mathcal{N}_\T(\alpha)} \varepsilon(\gamma)^2.
\end{equation}
\end{theorem}

\begin{theorem}\label{thm:mode_bound}
	With the same notation as in \cref{thm:hapod_bound} we have for each $\alpha \in
	\mathcal{N}_\T$ the following bound for the number of HAPOD modes:
\begin{equation}\label{eq:mode_bound}
	\Bigl|\HAPODSTDE(\alpha)\Bigr| \leq
	\Bigl|\POD\Bigl(\AccSnap_\alpha, \varepsilon(\alpha)\Bigr)\Bigr|.
\end{equation}	
\end{theorem}

In model reduction applications, the $\ell^2$-\emph{mean} approximation error is often the desired quantity
to optimize for, since in many cases neither the number of POD input vectors is known a priori (think of
adaptive time stepping schemes) nor the number of vectors which are to be approximated by the generated
POD space (i.e. the number of reduced model evaluations).
Thus, we want to define $\varepsilon$ such that the mean $\ell^2$-error is bounded by a desired
target tolerance $\varepsilon^*$, independently from the total number of input modes $|\Snap|$.
At the same time, the number of HAPOD output modes should not be much larger than the optimal quantity
$|\MPOD(\Snap, \varepsilon^*)|$, where
\begin{equation*}
   \MPOD(\Snap, \varepsilon^*) := \POD(\Snap, \sqrt{|\Snap|} \cdot \varepsilon^*).
\end{equation*}
In view of the above results, this motivates the following choice for $\varepsilon$,
where the parameter $\omega$ allows us to choose a trade-off between
efficiency of the HAPOD and the optimality of the resulting approximation space:

\begin{theorem} \label{cor:hapod_bound}
Using the same notation as in \cref{thm:hapod_bound}, let for $\varepsilon^* > 0$ the HAPOD
tolerances $\varepsilon(\rho_\T)$, $\varepsilon(\alpha)$, $\alpha \in \mathcal{N}_\T \setminus \{\rho_\T\}$
be given by:
\begin{equation*}
	\varepsilon(\rho_\T):= \sqrt{|\Snap|}\cdot \omega \cdot \varepsilon^*, \qquad
	\varepsilon(\alpha):= \sqrt{|\AccSnap_\alpha|}\cdot {(L_\T - 1)}^{-1/2}\cdot \sqrt{1 -
	\omega^2}\cdot
\varepsilon^*,
\end{equation*}
where $0 \leq \omega \leq 1$ is an arbitrary parameter.
Then we have the following bounds for the final $\ell^2$-mean approximation error and number of HAPOD modes:
\begin{equation}\label{eq:global_bounds}
	\frac{1}{|\Snap|} \sum_{s \in \Snap} \|s - P_{\rho_\T}(s)\|^2 \leq \varepsilon^{*2}
\quad\text{and}\quad
\Bigl|\HAPODSTDE\Bigr| \leq \Bigl|\MPOD(\Snap, \omega \cdot\varepsilon^*)\Bigr|.
\end{equation}
Moreover, the number of HAPOD modes at the intermediate stages $\alpha$ is
bounded by:
\begin{equation}\label{eq:local_bound}
	\Bigl|\HAPODSTDE(\alpha)\Bigr| \leq \Bigl|\MPOD(\AccSnap_\alpha, (L_\T - 1)^{-1/2} \cdot \sqrt{1-\omega^2} \cdot \varepsilon^*)\Bigr|.
\end{equation}
\end{theorem}

\begin{remark}\label{rem:hapod_efficiency}
	Note that the number of local POD modes $|\HAPODSTDE(\alpha)|$ determines the size of the input
	$\Snap_{\beta}$ for the next POD at the parent node $\beta$, and hence the effort required for its
	computation.
	Choosing a large $\omega \to 1$ will reduce the number of final HAPOD modes at the price of larger
	local PODs.
	A small $\omega \to 0$ will minimize the costs for computing the HAPOD in exchange for a larger number
	of final modes to guarantee the prescribed error bound.
\end{remark}
\begin{remark}\label{rem:hapod_bounded_by_nwidths}
	Since we consider the \emph{mean} square approximation error, it is possible for the bound \cref{eq:local_bound}
    that we have
	\[|\MPOD(\AccSnap_\alpha, \delta)| > |\MPOD(\Snap, \delta)|,\]
    where $\delta := (L_\T - 1)^{-1/2} \cdot  \sqrt{1 - \omega^2} \cdot \varepsilon^*$.
This might be the case when the principal directions of the snapshot set $\AccSnap_\alpha$ are underrepresented in
the full snapshot set $\Snap$.
However, if $N^\prime := \min \{ N \in \mathbb{N} \Bigm| d_N(\Snap) \leq \delta\}$, where
\begin{equation*}
d_N(\Snap) := \min_{\substack{X \subseteq V \text{ lin subsp.}\\ \dim X \leq N}}\, \max_{s \in \Snap} \, \|s - P_X(s)\|,
\end{equation*}
is the so-called Kolomogorov $N$-width of $\Snap$, and $X_{N^\prime}$ is a minimizer for $d_{N^\prime}(\Snap)$,
we always have
\[{|\AccSnap_\alpha|}^{-1}\sum_{s
\in \AccSnap_\alpha}\|s - P_{X_{N^\prime}}(s)\|^2 \leq \max_{s \in \AccSnap_\alpha} \|s - P_{X_{N^\prime}}(s)\|^2 \leq
\max_{s \in \Snap} \|s - P_{X_{N^\prime}}(s)\|^2 \leq \delta.\]
Thus, due to the optimality of the POD (\cref{thm:pod_error}) we have $|\MPOD(\AccSnap_\alpha, \delta)| \leq
N^\prime$, and the number of modes at $\alpha$ can be bounded by
\begin{equation}\label{eq:modes_vs_dn}
	\Bigl|\HAPODSTDE(\alpha)\Bigr| \leq \min \Bigl\{ N \in \mathbb{N} \Bigm| d_N(\Snap) \leq (L_\T - 1)^{-1/2} \cdot
    \sqrt{1-\omega^2} \cdot \varepsilon^*\Bigr\}.
\end{equation}
In many cases it is known theoretically or heuristically that $d_N(\Snap)$ shows rapid \mbox{(sub-)exponential} decay
for increasing $N$. In these cases, \cref{eq:modes_vs_dn} will be an effective upper bound for the number of local HAPOD
modes, independent of the chosen snapshot distribution $D$.
\end{remark}

\begin{remark}[Low-rank approximation of the snapshot mapping]
By additionally keeping track of the local right-singular vectors appearing in the
HAPOD algorithm, we easily obtain a low-rank approximation of the global
snapshot mapping $\AccSnapMap_\alpha: \mathbb{R}^{|\AccSnap_\alpha|} \to V$ defined in
\cref{def:hapod_notation}.
More precisely, by \cref{thm:main_recursion_lemma} and \cref{eq:hapod_error_estimate_in_proof}
we immediately have the rank-$\Bigl|\HAPODSTDE(\alpha)\Bigr|$ approximation:
\begin{equation}\label{eq:low_rank_approx}
    \| \AccSnapMap_\alpha - \Psi_\alpha \circ \widetilde{\Lambda}_\alpha^* \|_2^2  \ \leq
	\sum_{\gamma \in \mathcal{N}_\T(\alpha)} \varepsilon(\gamma)^2
\end{equation}
in the Hilbert-Schmidt (Frobenius) norm, with $\Psi_\alpha$,  $\widetilde{\Lambda}_\alpha$ 
given as in \cref{def:hapod_notation}.
\end{remark}

\subsection{Algorithmic Benefits}\label{sc:applic}
\cref{thm:hapod_bound,thm:mode_bound} show that, with an appropriate choice of local
error tolerances $\varepsilon$ (\cref{cor:hapod_bound}), the HAPOD produces
approximation spaces of a quality comparable to a POD with the same target
error tolerance.
At the same time, the HAPOD offers several benefits, which for problems
with fast decaying singular values can lead to dramatic speedups in computation time.

\paragraph{Reduced memory requirements}
If the input data for a POD cannot be kept completely in memory,
huge performance penalties are to be expected,
since for standard POD algorithms, repeated access of every snapshot vector is required.
If the data is kept on a mass storage device, the overall performance of the algorithm will usually be bounded by the data transfer speed.

For the HAPOD, at each node $\alpha$, only the vectors
$\Snap_\alpha$ are required as input to a local POD where,
typically, $|\Snap_\alpha| \ll |\Snap|$
so that $\Snap_\alpha$ can be kept completely in memory.

If only the POD, and not the snapshots themselves, is targeted by the computation,
the HAPOD can obtain the result without accessing mass storage altogether (cf.~\cref{sc:boltzmann}).
In particular, an incremental HAPOD of a time series may be computed even if the whole time series would not fit into memory
(cf.\ \cref{sc:special_cases,subsc:livecompr}).

\paragraph{Simple parallelization}
To compute the local POD at node $\alpha$, only the output of the PODs at the
child nodes $\mathcal{C}_\T(\alpha)$ is required.
In particular, for each $1 \leq l \leq L_\T$, all PODs at the nodes $\{\alpha \in \mathcal{N}_\T\ |\ {L}_\T(\alpha) = l\}$
can be computed in parallel without any communication, 
which is typically the bottleneck for distributed computations.
Intermediate results have to be communicated only vertically up the tree,
and the communicated data encompasses only low-rank quantities of computed
POD modes and singular values (cf.\ \cref{sc:special_cases,sc:gramian,sc:boltzmann}).

\paragraph{Generality}
The HAPOD can be applied using any pre-existing, optimized POD algorithm.
For instance, the HAPOD could be used to perform incremental data compression
for an MPI (Message Passing Interface) \cite{mpi} distributed model, where each sub-POD is computed via a parallelized
SVD algorithm.
In \cref{sc:boltzmann} we speed up the POD algorithm in \cref{def:method_of_snapshots} by exploiting the
block structure of the local Gramian similar to Brand's algorithm \cite{brand06}.

\paragraph{Lower algorithmic complexity}
A widely used, simple and reliable algorithm for POD computation
is to compute the eigenvalue decomposition of the Gramian to
$\Snap$ (cf.\ \cref{def:method_of_snapshots}).
In the case of $|\Snap| \ll d:=\dim(V)$, the Gramian computation dominates
the overall runtime for the algorithm with a computational complexity of 
$\mathcal{O}(|\Snap|^2d)$.
For larger snapshots sets $\Snap$ the quadratic increase in
complexity makes this method expensive in comparison to more advanced
algorithms (such as Lanczos or randomized methods \cite{chen09,halko11}), which scale only linearly in the number of snapshot vectors

Application of the HAPOD algorithm largely mitigates this issue.
In particular, for a balanced $n$-ary tree $\T$ with single vectors attached to
the leaves, the HAPOD using this POD algorithm
requires at most $\mathcal{O}(|\Snap|\log(|\Snap|)\widehat{N}^2d)$ operations for Gramian computation,
where $\widehat{N}:=\max_{\alpha \in \mathcal{N}_\T} |\HAPODSTDE(\alpha)|$ denotes the maximum
number of local output modes.
Assuming that the error tolerances $\varepsilon$ are
chosen according to \cref{cor:hapod_bound} for fixed $\varepsilon^*$, $\omega$, and assuming 
that the Kolmogorov widths $d_N(\Snap)$ are bounded for growing
$\Snap$, then, due to \cref{eq:modes_vs_dn}, $\widehat{N}$ will only depend on the depth $L_\T$ of $\T$.
If we furthermore assume that $d_N(\Snap)$ decays exponentially with increasing $N$,
we have $\widehat{N} = \mathcal{O}(\log(L_\T)) = \mathcal{O}(\log(\log(|\Snap|)))$.
Thus, the overall effort for computing the Gramians is reduced to
$\mathcal{O}(|\Snap|\log(|\Snap|)\log(\log(|\Snap|))^2d)$.

\subsection{Proofs of Main Theorems}\label{sc:proofs}

In this section we prove our main results (\cref{thm:hapod_bound,thm:mode_bound}).
We will require some additional notation:

\begin{definition}[Additional notation]\label{def:hapod_notation}
For each $\alpha \in \mathcal{N}_\T \setminus \mathcal{L}_\T$ fix an arbitrary enumeration
$\children{\alpha, 1},\ldots, \children{\alpha,|\children{\alpha}|}$ of $\children{\alpha}$.
For each $\alpha \in \mathcal{N}_\T$ we define mappings
\begin{equation*}\SnapMap_\alpha: \R^{|\Snap_\alpha|} \to V, \quad \Psi_\alpha: \R^{N_\alpha} \to V, \quad
\Lambda_\alpha: \R^{N_\alpha} \to \R^{|\Snap_\alpha|},
\end{equation*}
$N_\alpha:=|\HAPODSTDE(\alpha)|$, simultaneously recursively as follows:

As in \cref{eq:snapshot_map}, let $\SnapMap_\alpha$ map the $n$-th canonical basis
vector of $\R^{|\Snap_\alpha|}$ to the $n$-th element of $\Snap_\alpha$ for a given
enumeration of $\Snap_\alpha$. For $\alpha \in \mathcal{L}_\T$, the enumeration of
$\Snap_\alpha = D^{-1}(\{\alpha\})$ is chosen arbitrarily. For $\alpha \in \mathcal{N}_\T \setminus
\mathcal{L}_\T$, the enumeration is chosen such that the following compatibility relation is satisfied
\begin{equation}\label{eq:hapod_bound_SnapMap_compatibility}
    \SnapMap_\alpha = [\Psi_{\children{\alpha, 1}}, \ldots, \Psi_{\children{\alpha,
    |\children{\alpha}|}}].
\end{equation}
For $\varepsilon(\alpha) > 0$, let $\Psi_\alpha$, $\Lambda_\alpha$  be the linear mappings given by
\begin{equation*}
    \Psi_\alpha(e_n):= \sigma_n \cdot \varphi_n, \quad \Lambda_\alpha(e_n):= \lambda_n,
\end{equation*}
where $e_n$ is the $n$-th canonical basis vector of $\mathbb{R}^{N_\alpha}$ and $\sigma_n$, $\varphi_n$, $\lambda_n$
denote the $n$-th singular value, left singular vector and right singular vector of $\SnapMap_\alpha$.
Thus, $\Psi_\alpha \circ \Lambda_\alpha^*$ is the truncated SVD of $\SnapMap_\alpha$.
In particular, we have
\begin{equation}\label{eq:hapod_bound_truncated_svd}
    P_\alpha \circ \SnapMap_\alpha = \Psi_\alpha \circ \Lambda_\alpha^*, \qquad
    \Lambda_\alpha^* \circ \Lambda_\alpha = 1.
\end{equation}
    For $\varepsilon(\alpha) = 0$ (in which case $N_\alpha = |\Snap_\alpha|$), we simply let $\Psi_\alpha := \SnapMap_\alpha$ and let $\Lambda_\alpha$
    be the identity on $\R^{N_\alpha}$ such that \cref{eq:hapod_bound_truncated_svd} holds as well.

Note that since $\Snap_\alpha$ exactly consists of elements $\Psi_\beta(e_n)$ with $\beta \in \children{\alpha}$,
$1 \leq n \leq N_\beta$, it is clear that \cref{eq:hapod_bound_SnapMap_compatibility} can always be satisfied.

Finally, we define cumulative mappings
$\AccSnapMap_\alpha, \ProjSnapMap_\alpha: \R^{|\AccSnap_\alpha|} \to V$,
$\widetilde{\Lambda}_\alpha: \R^{N_\alpha} \to \R^{|\AccSnap_\alpha|}$ recursively as
\begin{equation*}
	\AccSnapMap_\alpha := \SnapMap_\alpha, \qquad \ProjSnapMap_\alpha:= \SnapMap_\alpha,
	\qquad \widetilde{\Lambda}_\alpha := \Lambda_\alpha,
\end{equation*}
for $\alpha \in \mathcal{L}_\T$ and
\begin{gather*}
	\AccSnapMap_\alpha := [\AccSnapMap_{\children{\alpha, 1}}, \ldots, \AccSnapMap_{\children{\alpha,
	|\children{\alpha}|}}], \\
	\widetilde{\Lambda}_\alpha := \operatorname{diag}(\widetilde{\Lambda}_{\children{\alpha, 1}},\ \ldots,\
	\widetilde{\Lambda}_{\children{\alpha, |\children{\alpha}|}})
	                              \circ \Lambda_\alpha,
		\\
		\ProjSnapMap_\alpha := [P_{\children{\alpha, 1}}\circ \ProjSnapMap_{\children{\alpha, 1}},\ \ldots,\ 
				P_{\children{\alpha, |\children{\alpha}|}} \circ \ProjSnapMap_{\children{\alpha,
			|\children{\alpha}|}}],
\end{gather*}
for all $\alpha \in \mathcal{N}_\T \setminus \mathcal{L}_\T$.
Similar to the definition of $\SnapMap_\alpha$, the map $\AccSnapMap_\alpha$ is of the
form \cref{eq:snapshot_map} with respect to a specific enumeration of $\AccSnap_\alpha$.
\end{definition}

	As a first step towards the proof of our main theorems, we will extend the decomposition \cref{eq:hapod_bound_truncated_svd} to
    the accumulated mapping of projected snapshots $\ProjSnapMap_\alpha$:

\begin{lemma}\label{thm:main_recursion_lemma}
	With the same notation as in \cref{def:hapod_notation} we have for all $\alpha \in \mathcal{N}_\T$:
	\begin{equation}\label{eq:hapod-bound-lemma}
		P_\alpha \circ \ProjSnapMap_\alpha = \Psi_\alpha \circ \widetilde{\Lambda}_\alpha^*, \qquad \widetilde{\Lambda}_\alpha^* \circ
		\widetilde{\Lambda}_\alpha = 1.
	\end{equation}
	In particular, it follows for $\alpha \in \mathcal{N}_\T \setminus \mathcal{L}_\T$ that
	\begin{equation}\label{eq:hapod-bound-lemma-b}
		  \ProjSnapMap_\alpha = \SnapMap_\alpha  \circ \operatorname{diag}(\widetilde{\Lambda}^*_{\children{\alpha, 1}},
						\ldots, \widetilde{\Lambda}^*_{\children{\alpha, |\children{\alpha}|}}).
	\end{equation}
\end{lemma}

\begin{proof}
We show the claim via induction over $\T$.
To this end, first note that for $\alpha \in \mathcal{L}_\T$,
\cref{eq:hapod-bound-lemma} is precisely \cref{eq:hapod_bound_truncated_svd} by definition of
$\ProjSnapMap_\alpha, \widetilde{\Lambda}_\alpha$.
For $\alpha \in \mathcal{N}_\T \setminus \mathcal{L}_\T$, we obtain using the induction hypothesis, the definition of
$\SnapMap_\alpha$ and \cref{eq:hapod_bound_truncated_svd}:
\begin{align*}
P_\alpha\circ\ProjSnapMap_\alpha &= P_\alpha \circ [P_{\children{\alpha, 1}}\circ \ProjSnapMap_{\children{\alpha, 1}},\ \ldots ,\ 
				P_{\children{\alpha, |\children{\alpha}|}} \circ \ProjSnapMap_{\children{\alpha,
			|\children{\alpha}|}}] \\
	&= P_\alpha \circ [\Psi_{\children{\alpha, 1}} \circ \widetilde{\Lambda}^*_{\children{\alpha, 1}}, \ldots ,
	\Psi_{\children{\alpha, |\children{\alpha}|}} \circ \widetilde{\Lambda}^*_{\children{\alpha, |\children{\alpha}|}}] \\
	&= P_\alpha \circ \SnapMap_\alpha \circ \operatorname{diag}(\widetilde{\Lambda}^*_{\children{\alpha, 1}},
	\ldots , \widetilde{\Lambda}^*_{\children{\alpha, |\children{\alpha}|}}) \\
	&= \Psi_\alpha \circ \Lambda_\alpha^* \circ \operatorname{diag}(\widetilde{\Lambda}^*_{\children{\alpha, 1}},
	\ldots , \widetilde{\Lambda}^*_{\children{\alpha, |\children{\alpha}|}}) \\
	&= \Psi_\alpha \circ \widetilde{\Lambda}_\alpha^*.
\end{align*}
Moreover:
\begin{equation*}
\widetilde{\Lambda}_\alpha^* \circ \widetilde{\Lambda}_\alpha
  = \Lambda_\alpha^* \circ \operatorname{diag}(\widetilde{\Lambda}_{\children{\alpha, 1}}^* \circ
	  \widetilde{\Lambda}_{\children{\alpha, 1}},\ \ldots,\
                               \widetilde{\Lambda}_{\children{\alpha, |\children{\alpha}|}}^* \circ
		       \widetilde{\Lambda}_{\children{\alpha, |\children{\alpha}|}}) \circ \Lambda_\alpha
			       = 1.
\end{equation*}
Thus, \cref{eq:hapod-bound-lemma} is proved, and we have
\begin{align*}
  \ProjSnapMap_\alpha &=
    [P_{\children{\alpha, 1}}\circ \ProjSnapMap_{\children{\alpha, 1}},\ \ldots,\ 
			P_{\children{\alpha, |\children{\alpha}|}} \circ \ProjSnapMap_{\children{\alpha,
		|\children{\alpha}|}}] \\
    &= [\Psi_{\children{\alpha, 1}} \circ \widetilde{\Lambda}_{\children{\alpha, 1}}^*, \ldots,
     \Psi_{\children{\alpha, |\children{\alpha}|}} \circ \widetilde{\Lambda}_{\children{\alpha, |\children{\alpha}|}}^*] \\
     &= \SnapMap_\alpha  \circ \operatorname{diag}(\widetilde{\Lambda}^*_{\children{\alpha, 1}},
	\ldots, \widetilde{\Lambda}^*_{\children{\alpha, |\children{\alpha}|}}).
\end{align*}
\end{proof}

As a final preparatory step, we show the following orthogonality lemma:

\begin{lemma}\label{thm:orthogonality_lemma}
	With the same notation as in \cref{def:hapod_notation} we have for all $\alpha \in \mathcal{N}_\T$
	and arbitrary continuous linear maps $X, Y: V \to V$:
	\begin{equation} \label{eq:hapod-bound-orthogonality}
		(X\circ(\AccSnapMap_\alpha - \ProjSnapMap_\alpha),\, Y\circ \ProjSnapMap_\alpha)_2 = 0,
	\end{equation}
	where $(A, B)_2$ is the Hilbert-Schmidt inner product given by $\Tr(A^*B)$.
\end{lemma}

\begin{proof}
We prove the claim again via induction over $\T$. 	
For $\alpha \in \mathcal{L}_\T$ the statement is obvious since $\AccSnapMap_\alpha = \SnapMap_\alpha =
\ProjSnapMap_\alpha$.
For $\alpha \in \mathcal{N}_\T \setminus \mathcal{L}_\T$, we have
\begin{multline}\label{eq:hapod-bound-orthogonality-proof}
	(X\circ(\AccSnapMap_\alpha - \ProjSnapMap_\alpha),\, Y\circ \ProjSnapMap_\alpha)_2 \\
\begin{aligned}
	&= \sum_{\beta \in \mathcal{C}_\T(\alpha)} (X\circ(\AccSnapMap_\beta - P_\beta\circ
	\ProjSnapMap_\beta),\, Y\circ P_\beta \circ \ProjSnapMap_\beta)_2 \\
	&= \sum_{\beta \in \mathcal{C}_\T(\alpha)} (X\circ(\AccSnapMap_\beta - \ProjSnapMap_\beta),\,
Y\circ P_\beta\circ \ProjSnapMap_\beta)_2 \\ & \qquad\qquad\qquad\qquad+
	\sum_{\beta \in \mathcal{C}_\T(\alpha)} (X\circ (1 - P_\beta)\circ \ProjSnapMap_\beta,\,
	Y\circ P_\beta \circ\ProjSnapMap_\beta)_2.
\end{aligned}
\end{multline}
The first sum on the right-hand side of \cref{eq:hapod-bound-orthogonality-proof} vanishes by induction hypothesis
(with $Y:=Y \circ P_\beta$). 
To handle the second sum note that for $\beta \in \mathcal{N}_\T \setminus \mathcal{L}_\T$, $\varepsilon(\beta) > 0$ we can use \cref{eq:hapod-bound-lemma-b} to write:
\begin{align*}
(1 - P_\beta)\circ\ProjSnapMap_\beta &= (1 - P_\beta) \circ \SnapMap_\beta \circ
	\operatorname{diag}(\widetilde{\Lambda}^*_{\children{\beta, 1}},
	\ldots, \widetilde{\Lambda}^*_{\children{\beta, |\children{\beta}|}}) \\
	&= \Psi^c_\beta \circ \Lambda_\beta^{c*} \circ
	\operatorname{diag}(\widetilde{\Lambda}^*_{\children{\beta, 1}},
	\ldots, \widetilde{\Lambda}^*_{\children{\beta, |\children{\beta}|}}),
\end{align*}
where $\Psi_\beta^c: \mathbb{R}^{|\Snap_\beta| - N_\beta} \to V, \Lambda_\beta^c: \mathbb{R}^{|\Snap_\beta| -
N_\beta} \to \mathbb{R}^{|\Snap_\beta|}$ map the $k$-th canonical basis vector to the $(N_\beta + k)$-th
scaled left (unscaled right) singular vector of $\SnapMap_\beta$.
In particular, $\Lambda_\beta^* \circ \Lambda_\beta^c = 0$.
Using \cref{eq:hapod-bound-lemma} and the invariance of the trace under cyclic permutations, we obtain:
\begin{multline*}
(X\circ(1 - P_\beta)\circ\ProjSnapMap_\beta,\, Y\circ P_\beta\circ \ProjSnapMap_\beta)_2 \\
\begin{aligned}
	&= \Tr(\{(1 - P_\beta)\circ \ProjSnapMap_\beta\}^*\circ X^*\circ Y \circ P_\beta\circ \ProjSnapMap_\beta) \\
	&= \Tr(X^*\circ Y\circ P_\beta\circ \ProjSnapMap_\beta\circ \{(1 - P_\beta)\circ \ProjSnapMap_\beta\}^*) \\
	&= \Tr(X^*\circ Y \circ \Psi_\beta \circ \Lambda_\beta^* \circ
	\operatorname{diag}(\widetilde{\Lambda}^*_{\children{\beta, 1}},
	\ldots, \widetilde{\Lambda}^*_{\children{\beta, |\children{\beta}|}}) \\
	&\qquad\qquad\qquad\qquad\qquad\qquad\qquad \circ
	\operatorname{diag}(\widetilde{\Lambda}_{\children{\beta, 1}}, \ldots, \widetilde{\Lambda}_{\children{\beta,
|\children{\beta}|}}) \circ \Lambda_\beta^c \circ \Psi_\beta^{c*}) \\
&= \Tr(X^*\circ Y \circ \Psi_\beta \circ \{\Lambda_\beta^* \circ \Lambda_\beta^c\} \circ \Psi_\beta^c)
	= 0.
\end{aligned}
\end{multline*}
The same line of argument holds for $\beta \in \mathcal{L}_\T$, where we have $(1 - P_\beta)\circ
\ProjSnapMap_\beta = \Psi_\beta^c \circ \Lambda_\beta^{c*}$.
Since for $\varepsilon(\beta) = 0$ we trivially have $1 - P_\beta = 0$, we see that the second sum in \cref{eq:hapod-bound-orthogonality-proof} always vanishes, proving
the claim.
\end{proof}

\begin{proof}[Proof of \cref{thm:hapod_bound}]
First note that,
due to the best approximation property of the orthogonal projection $P_{\alpha}$ we have:
\begin{align*}
\sum_{s \in \AccSnap_\alpha} \|s - P_{\alpha}(s)\|^2  &=
 \sum_{n = 1}^{|\AccSnap_\alpha|} \|\AccSnapMap_{\alpha}(e_n) - P_{\alpha}(\AccSnapMap_{\alpha}(e_n))\|^2 \\
 &\leq \sum_{n = 1}^{|\AccSnap_\alpha|} \|\AccSnapMap_{\alpha}(e_n) - P_{\alpha}(\ProjSnapMap_{\alpha}(e_n))\|^2  \\
 &= \|\AccSnapMap_{\alpha} - P_{\alpha}\circ\ProjSnapMap_{\alpha}\|^2_2,
\end{align*}
where $\|A\|_2 = \sqrt{(A, A)_2} = \sqrt{\Tr(A^*A)}$ denotes the Hilbert-Schmidt norm of $A$.
Thus, the theorem is proven if we can show the that for all $\alpha \in \mathcal{N}_\T$ the
following estimate holds:
\begin{equation}\label{eq:hapod_error_estimate_in_proof}
\|\AccSnapMap_\alpha - P_\alpha\circ\ProjSnapMap_\alpha\|_2^2 \leq
\sum_{\gamma \in \mathcal{N}_\T(\alpha)} \varepsilon(\gamma)^2.
\end{equation}
We show \cref{eq:hapod_error_estimate_in_proof} again via induction over $\T$.
For $\alpha \in \mathcal{L}_\T$ we immediately have:
\begin{equation*}
 \|\AccSnapMap_\alpha - P_\alpha\circ\ProjSnapMap_\alpha\|_2^2
 = \| \SnapMap_\alpha - P_\alpha \circ \SnapMap_\alpha \|_2^2 \leq \varepsilon(\alpha)^2 =
 \sum_{\gamma \in \mathcal{N}_\T(\alpha)} \varepsilon(\gamma)^2,
\end{equation*}
according to \cref{def:pod}.

Now, let us assume that \cref{eq:hapod_error_estimate_in_proof} holds for all $\beta \in
\children{\alpha}$ for some $\alpha \in \mathcal{N}_\T \setminus \mathcal{L}_\T$.
Using \cref{thm:orthogonality_lemma} with $Y = I - P_\alpha$, we have 
\begin{equation*}
\|\AccSnapMap_\alpha - P_\alpha\circ\ProjSnapMap_\alpha\|_2^2
= \| \AccSnapMap_\alpha - \ProjSnapMap_\alpha + (I - P_\alpha) \circ
\ProjSnapMap_\alpha \|_2^2 = \| \AccSnapMap_\alpha - \ProjSnapMap_\alpha \|_2^2 + \| (I - P_\alpha) \circ
\ProjSnapMap_\alpha \|_2^2.
\end{equation*}
Using the induction hypothesis, we can bound the first summand by:
\begin{equation*}
\begin{aligned}
\| \AccSnapMap_\alpha - \ProjSnapMap_\alpha \|_2^2
&= \sum_{\beta \in \children{\alpha}} \|\AccSnapMap_\beta - P_\beta\circ\ProjSnapMap_\beta\|_2^2 \\
&\leq \sum_{\beta \in \children{\alpha}} \sum_{\gamma \in \mathcal{N}_\T(\beta)} \varepsilon(\gamma)^2 \\
&= \sum_{\gamma \in \mathcal{N}_\T(\alpha) \setminus \{\alpha\}} \varepsilon(\gamma)^2.
\end{aligned}
\end{equation*}

To bound the second summand, we use \cref{thm:main_recursion_lemma}, the fact that $\|T\circ S\|_2 \leq \|T\|_2\cdot\|S\|$
(for arbitrary $T$, $S$) and \cref{def:pod} to obtain:
\begin{align*}
\| (I - P_\alpha) \circ \ProjSnapMap_\alpha \|_2^2
   &=\| (I - P_\alpha) \circ \SnapMap_\alpha \circ \operatorname{diag}(\widetilde{\Lambda}^*_{\children{\alpha, 1}},
	\ldots, \widetilde{\Lambda}^*_{\children{\alpha, |\children{\alpha}|}})\|_2^2 \\
 &\leq\| (I - P_\alpha) \circ \SnapMap_\alpha \|_2^2 \cdot \|\operatorname{diag}(\widetilde{\Lambda}^*_{\children{\alpha, 1}}, 
	\ldots, \widetilde{\Lambda}^*_{\children{\alpha, |\children{\alpha}|}})\|^2 \\
	&\leq \varepsilon(\alpha)^2.
\end{align*}
Thus, \cref{eq:hapod_error_estimate_in_proof} follows, which completes the proof.
\end{proof}

\begin{proof}[Proof of \cref{thm:mode_bound}]
For $\alpha \in \mathcal{L}_\T$ there is nothing to show, so let us assume that
$\alpha \in \mathcal{N}_\T \setminus \mathcal{L}_\T$.
According to \cref{thm:main_recursion_lemma},
$\ProjSnapMap_\alpha$ and $\SnapMap_\alpha$ have the same singular values.
Thus, with $\ProjSnap_{\alpha}:=\{\ProjSnapMap_\alpha(e_n)\ |\ 1 \leq n \leq
|\AccSnap_\alpha|\}$ we have:
\begin{equation*}
	|\HAPODSTDE(\alpha)| = |\POD(\Snap_\alpha, \varepsilon(\alpha))|
	= |\POD(\ProjSnap_\alpha, \varepsilon(\alpha))|.
\end{equation*}
Let $\widetilde{P}_\alpha$ be the orthogonal projection onto the linear span of the modes selected by
$\POD(\AccSnap_\alpha, \varepsilon(\alpha))$.
Due to \cref{thm:orthogonality_lemma} with $X = Y = 1 - \widetilde{P}_\alpha$, we have:
\begin{align*}
	\varepsilon(\alpha)^2
	&\geq \|(1 - \widetilde{P}_\alpha)\circ \AccSnapMap_\alpha\|_2^2 \\
	 &= \|(1 - \widetilde{P}_\alpha)\circ \ProjSnapMap_\alpha \|_2^2 + 
	    \|(1 - \widetilde{P}_\alpha)\circ (\AccSnapMap_\alpha - \ProjSnapMap_\alpha) \|_2^2 \\
	 &\geq \|(1 - \widetilde{P}_\alpha)\circ \ProjSnapMap_\alpha\|_2^2.
\end{align*}
According to \cref{def:pod} and due to the optimality of the POD we therefore have
\begin{equation*}
	|\POD(\ProjSnap_\alpha, \varepsilon(\alpha))| \leq
	|\POD(\AccSnap_\alpha, \varepsilon(\alpha))|,
\end{equation*}
which concludes the proof.
\end{proof}

\begin{proof}[Proof of \cref{cor:hapod_bound}]
   According to \cref{thm:hapod_bound} we have 
\begin{align*}
    \sum_{s \in \Snap} \|s - P_{\rho_\T}(s)\|^2 &
    \leq
        |\Snap|\cdot \omega^2 \cdot \varepsilon^{*2} +
        \sum_{l=1}^{L_\T - 1}\sum_{\substack{\gamma \in \mathcal{N}_\T\\L_\T(\gamma) = l}} |\AccSnap_\gamma|\cdot {(L_\T - 1)}^{-1}\cdot (1 -
        \omega^2)\cdot \varepsilon^{*2}\\
    &\leq
        |\Snap|\cdot \omega^2 \cdot \varepsilon^{*2} +
        \sum_{l=1}^{L_\T - 1} |\Snap|\cdot {(L_\T - 1)}^{-1}\cdot (1 -
        \omega^2)\cdot \varepsilon^{*2}\\
    &= |\Snap|\cdot \varepsilon^{*2}.
\end{align*}
The stated bounds for the number of HAPOD modes follow directly from \cref{thm:mode_bound} and the
definition of $\MPOD$.
\end{proof}

\section{Numerical Results}\label{sc:numex}
To demonstrate the applicability of the HAPOD, three numerical examples comparing the POD with the HAPOD are presented and evaluated in terms of accuracy and complexity.
The first two experiments are implemented in the Matlab language and performed using Octave \cite{octave}.
For the POD and HAPOD\footnote{Internally the HAPOD implementation uses the same POD method as the plain POD.}, the built-in SVD of Octave is utilized,
which in turn uses LAPACK \cite{lapack}.
The third experiment is implemented in Python using the POD implementation of the \mbox{pyMOR} library \cite{pymor},
which utilizes the method of snapshots by SciPy's \cite{scipy} symmetric eigenvalue computation, also via LAPACK.

\subsection{Incremental Data Compression} \label{subsc:livecompr}
\begin{figure}[t] \centering

 \begin{subfigure}[t]{.48\textwidth}

\begin{tikzpicture}
  \begin{loglogaxis}
    [max space between ticks=70cm,
     line width=1pt,
     major tick length=2mm,
     every axis plot/.style={line width=1pt},
     every tick/.append style={line width=1pt},
     mark size=2pt,
     x dir=reverse,
     xtick={1.0,0.1,0.01,0.001},
     ytick={1.0,0.1,0.01,0.001},
     xlabel={Prescribed Mean Proj. Error},
     ylabel={Mean Projection Error}]
    \addplot+ [mark options={}] table[x index=0, y index=1] {numex1a.dat}; 
    \addplot+ [mark options={}] table[x index=0, y index=2] {numex1a.dat}; 
    \addplot [color=black] coordinates {(1,1) (0.1,0.1) (0.01,0.01) (0.001,0.001)};
    \legend{POD, HAPOD, $\varepsilon^*$}
  \end{loglogaxis}
\end{tikzpicture}

\caption{Actual $\ell^2$-mean projection error of POD and incremental HAPOD computation for prescribed errors $\varepsilon^*$.}
  \label{fig:numex1a1}
 \end{subfigure}
 \hfill
 \begin{subfigure}[t]{.48\textwidth}

\begin{tikzpicture}
  \begin{semilogxaxis}
    [legend pos=south east,
     max space between ticks=70cm,
     line width=1pt,
     major tick length=2mm,
     every axis plot/.style={line width=1pt},
     every tick/.append style={line width=1pt},
     mark size=2pt,
     x dir=reverse,
     xtick={1.0,0.1,0.01,0.001},
     ytick={10,20,30,40,50,60,70,80,90,100},
     xlabel={Prescribed Mean Proj. Error},
     ylabel={Number of Modes}]
    \addplot+ [mark options={}] table[x index=0, y index=1] {numex1b.dat}; 
    \addplot+ [mark options={}] table[x index=0, y index=2] {numex1b.dat}; 
    \addplot+ table[x index=0, y index=4] {numex1b.dat};
    \addplot+ [mark options={}] table[x index=0, y index=5] {numex1b.dat};
    \legend{POD, HAPOD, Bound, Intermed.}
  \end{semilogxaxis}
\end{tikzpicture}

\caption{Number of resulting POD and HAPOD modes, bound \cref{eq:global_bounds} for
	 number of HAPOD modes at output node $\rho_\T$
	 and maximum number of intermediate HAPOD output modes \cref{eq:local_bound}.}
  \label{fig:numex1a2}
 \end{subfigure}

 \caption{Approximation error and mode counts vs.\ prescribed error tolerance 
 for the data compression example with state-space dimension $N=500$ (cf.\ \cref{subsc:livecompr}).}
 \label{fig:numex1a}
\end{figure}
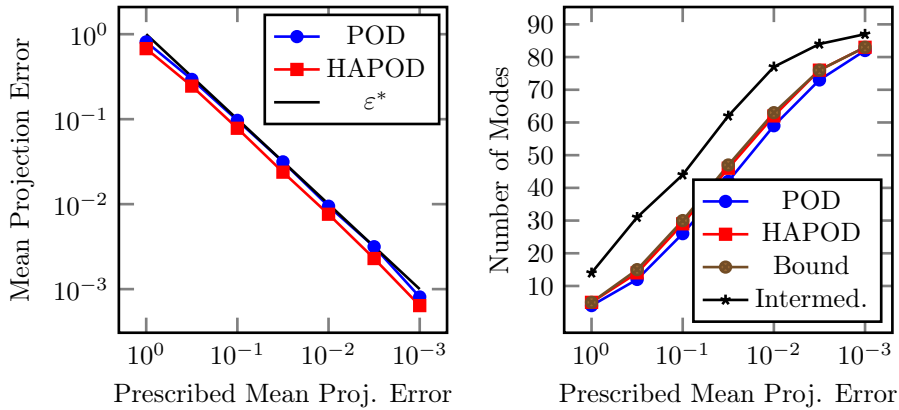

\begin{figure}[t] \centering

 \begin{subfigure}[t]{.48\textwidth}
 \centering
  \includegraphics[width=0.95\textwidth]{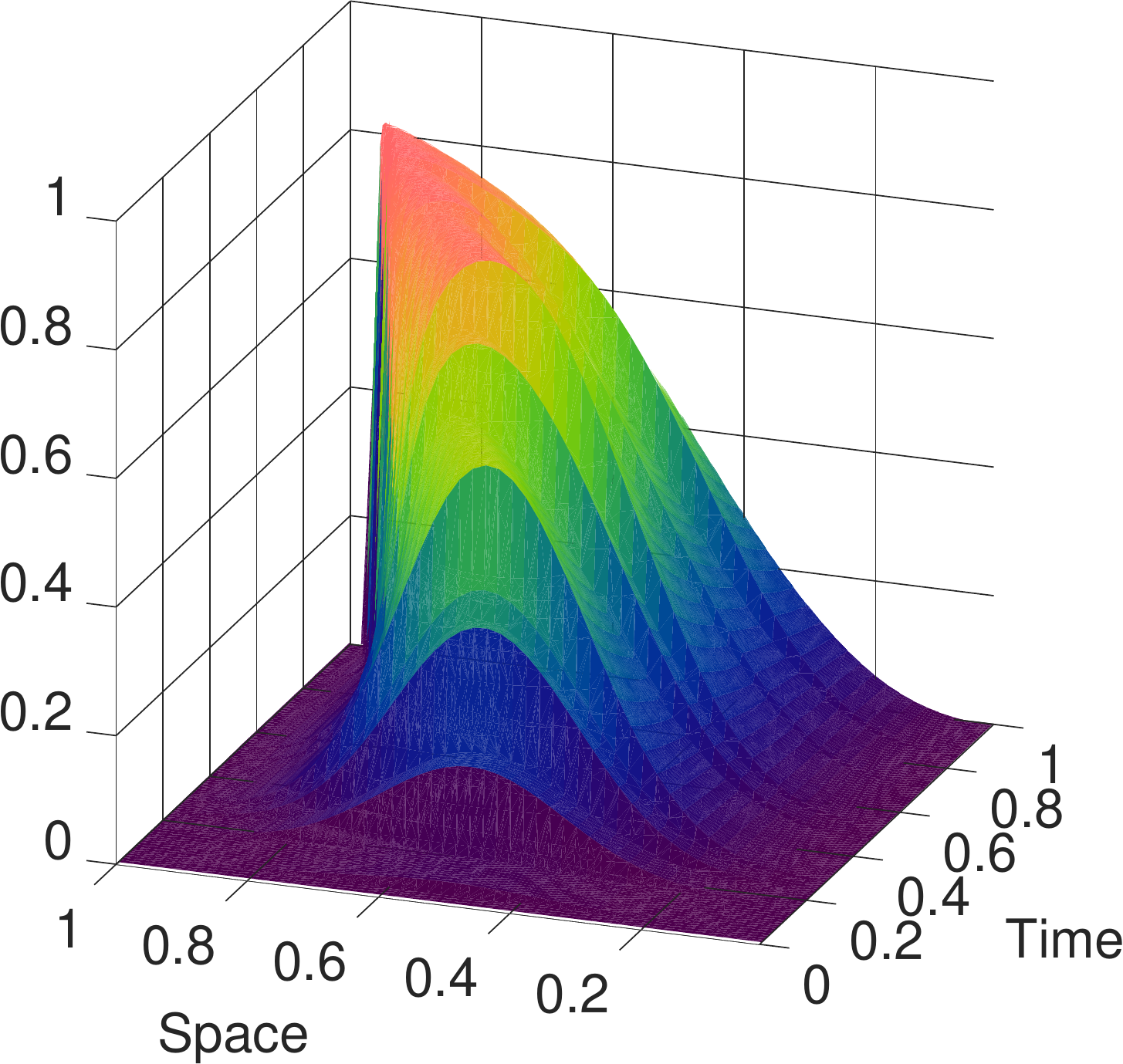}
  \caption{Visualization of the temporal evolution of the Burgers equation example.}
  \label{fig:numex1b1}
 \end{subfigure}
\hfill
 \begin{subfigure}[t]{.48\textwidth}
\begin{tikzpicture}
  \begin{semilogxaxis}
    [legend style={at={(0.97,0.5)},anchor=east},
     max space between ticks=70cm,
     line width=1pt,
     major tick length=2mm,
     every axis plot/.style={line width=1pt},
     every tick/.append style={line width=1pt},
     mark size=2pt,
     x dir=reverse,
     xtick={1.0,0.1,0.01,0.001},
     ytick={200,400,600,800},
     xlabel={Prescribed Mean Proj. Error},
     ylabel={Computational Time [s]}]
    \addplot+ [mark options={}] table[x index=0, y index=1] {numex1c.dat}; 
    \addplot+ [mark options={}] table[x index=0, y index=2] {numex1c.dat}; 
    \legend{POD, HAPOD}
  \end{semilogxaxis}
\end{tikzpicture}

  \caption{Computational time for POD and incremental HAPOD with state-space dimension $N=500$.}
  \label{fig:numex1c1}
 \end{subfigure}

 \caption{Solution visualization and computational time vs.\ prescribed error for the data compression example
 (cf.~\cref{subsc:livecompr}).}
 \label{fig:numex1b}
\end{figure}

The first numerical experiment compares the POD and HAPOD through compressing a trajectory of a randomly excited system.
As an underlying system, a forced one-dimensional inviscid Burgers equation is chosen:
\begin{align*}
 \partial_t z(x,t) + z(x,t) \cdot \partial_x z(x,t) &= b(x, t), &\hspace*{-1cm}(x, t) &\in (0,1) \times (0,1), \\
 z(x,0) &= 0, &\hspace*{-1cm} x &\in [0, 1],\\
 z(0,t) &= 0, &\hspace*{-1cm} t &\in [0, 1], 
\end{align*}
with force term $b \in L^2([0,1] \times [0,1])$.
A spatial discretization using a conservative finite difference upwind scheme with $N=500$ equidistant nodes yields a system of nonlinear ordinary differential equations in time \cite{Lev90}:
\begin{align*}
 \dot{z}(t) = A (z(t) \circ z(t)) + B u(t),
\end{align*}
with $\circ$ denoting the element-wise Hadamard product.
The experiment runs with constant temporal resolution $h = 10^{-4}$ resulting in $10^4$ explicit Euler time steps.
As forcing term, a scaled Gaussian bell curve $b(x,t) = u(t) \exp(-\frac{1}{20}(x-\frac{1}{2})^2)$ is chosen
with a time-dependent coefficient $u(t)$ which is $99.9\%$ of all time steps zero, but at random instances over the
whole time interval for $0.1\%$ of all time steps a constant value sampled from the uniform random distribution in the
interval $[0,\frac{1}{5}]$.
The full order model evolution is visualized in \cref{fig:numex1b1}.

An incremental HAPOD is performed as described in \cref{sc:special_cases}
to extract the dominant modes for different accuracies on a subdivision of the full time series into one-hundred uniform length blocks, 
of which results are compared to a POD over the whole time series.
The local error tolerances $\varepsilon$ are chosen according to \cref{cor:hapod_bound}
with $\omega = 0.75$. 
The computation is conducted on a Raspberry Pi\footnote{Rasperry Pi Model 1B: ARMv6-CPU $700$MHz, \textbf{512MB RAM}, see also: \mbox{\url{http://www.raspberrypi.org/products/model-b}}.}
single board computer device, which is a memory limited device, comparable to embedded or power-aware environments.

In \cref{fig:numex1a1}, the $\ell^2$-mean projection error \cref{eq:global_bounds} for the prescribed accuracies of
$\varepsilon^* \in \{10^0,10^{-1/2},10^{-1},\dots,10^{-3}\}$ is depicted.
Due to shock formation in the solution, a relatively large number of POD modes is required for accurate
approximation. 
Thus, in view of the low spatial resolution,
the prescribed errors are chosen in a manner to suppress effects of the discretization error in the results.
The approximation error of the POD and the incremental HAPOD decay very similarly in rate and magnitude.
In terms of the number of modes, \cref{fig:numex1a2} shows that also the number of final HAPOD modes increases with the same rate as the classic POD.
The HAPOD requires at most four additional modes,
and the mode bound \eqref{eq:global_bounds} overestimates the number of HAPOD modes by at most one. At most $15$ additional output modes
are generated at the intermediate HAPOD steps.
\begin{figure}[t] \centering

\begin{subfigure}[t]{.48\textwidth}

\begin{tikzpicture}
  \begin{semilogyaxis}
    [legend pos=south east,
     max space between ticks=70cm,
     line width=1pt,
     major tick length=2mm,
     every axis plot/.style={line width=1pt},
     every tick/.append style={line width=1pt},
     mark size=2pt,
     ytick={10,100,1000,10000},
     xlabel={State Dimension},
     ylabel={Computational Time [s]}]
    \addplot+ [mark options={}] table[x index=0, y index=1] {numex1z.dat}; 
    \addplot+ [mark options={}] table[x index=0, y index=2] {numex1z.dat}; 
    \legend{POD, HAPOD}
  \end{semilogyaxis}
\end{tikzpicture}

  \caption{Computational time for POD and incremental HAPOD with a prescribed error $\varepsilon^*=10^{-3/2}$ vs.\ different state-space
    dimensions.}
  \label{fig:numex1c2}
 \end{subfigure}
\hfill
 \begin{subfigure}[t]{.48\textwidth}
\begin{tikzpicture}
    \begin{loglogaxis}[at={(0.,0.)},
           anchor=east,
           line width=1pt,
           major tick length=2mm,
           every axis plot/.style={line width=1pt},
           every tick/.append style={line width=1pt},
           every y tick/.style={red},
           every y tick label/.append style = {red},
           mark size=2pt,
           axis y line*=right,
           ylabel near ticks, yticklabel pos=right,
           y axis line style={red,thick},
           xlabel={Block Size},
           ylabel={\textcolor{red}{Time [s]}},
           ylabel shift = -20pt,
           xtick={10,100,1000},
           ytick={100,1000},
           xshift=-50pt,
           yshift=0pt]

      \addplot+ [mark options={}, red, mark=x] table[x index=0, y index=3] {numex1y.dat}; \label{pl_time}
     \end{loglogaxis}
    \begin{semilogxaxis}[at={(0.,0.)},
           anchor=east,
           legend style={at={(0.03,0.52)},anchor=west,fill=none,draw=none},
           line width=1pt,
           major tick length=2mm,
           every axis plot/.style={line width=1pt},
           every tick/.append style={line width=1pt},
           every y tick/.style={blue},
           every y tick label/.append style = {blue},
           mark size=2pt,
           xshift=-50pt,
           yshift=0pt,
           xlabel={},
           ylabel={\textcolor{blue}{Number of Modes}},
           ylabel near ticks,
           y axis line style={blue,thick},
           axis x line=none,
           axis y line*=none
      ]
      \addlegendimage{/pgfplots/refstyle=pl_time}\addlegendentry{Runtime}
      \addplot+ [mark options={}, blue, mark=*] table[x index=0, y index=1] {numex1y.dat}; 
      \addplot+ [mark options={}, blue, mark=square*] table[x index=0, y index=2] {numex1y.dat}; 
      \addlegendentry{Final}
      \addlegendentry{Intermed.}
     \end{semilogxaxis}
\end{tikzpicture}

  \caption{Number of final HAPOD modes and maximum number of intermediate modes as well as computational time
  for varying input data block sizes,
  $\varepsilon^* = 10^{-3/2}$ and $N = 500$.}
  \label{fig:numex1b2}
 \end{subfigure}

 \caption{Computational time and mode number vs.\ state dimension and block size (the number of snapshots in a leaf node)
 for the data compression example (cf.\ \cref{subsc:livecompr}).}
 \label{fig:numex1c}
\end{figure}
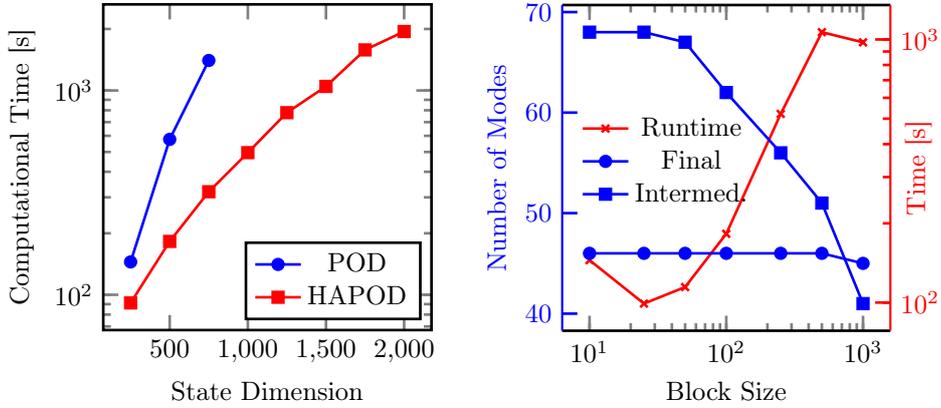

The time consumption is plotted in \cref{fig:numex1c1} for the different $\varepsilon^*$.
Since the used POD implementation fully factorizes the given input data,
the required computational time for the POD is (almost) constant for different accuracies.
The incremental HAPOD time requirements increase with higher accuracies,
yet for all tested $\varepsilon^*$ the HAPOD requires less time than the POD.
\cref{fig:numex1c2} shows the computational time for the POD and incremental HAPOD for varying state-space dimension \linebreak\mbox{$N=\{250, 500, 750, 1000, 1250, 1500, 1750, 2000\}$}, but fixed prescribed approximation error. 
For $N > 750$ the regular POD's memory requirements exceed the device capabilities, while the incremental HAPOD is still computable.

Furthermore, the dependence of the number of final HAPOD modes and intermediate modes together with the required
computational time is compared for varying block sizes in \cref{fig:numex1b2}. 
While the number of final modes stays almost constant, a smaller block size reduces the
computational time at the expense of a slightly larger number of intermediate modes.
This demonstrates the HAPOD's configurable trade-off between memory and computation time:
One can reduce the computational time by using smaller data partitions,
but has to take into account higher memory consumption for the intermediate modes;
on the other hand by enlarging the block partition size, less memory is consumed during the computation,
yet the computational time is increased.

\subsection{Distributed Empirical Cross Gramian}\label{sc:gramian}
The second numerical experiment compares the POD with the distributed HAPOD computation (cf.\ \cref{sc:special_cases})
in terms of the model reduction error resulting from the respective output modes.
Given a linear state-space control system with the same number of inputs and outputs $\dim(u(t)) = \dim(y(t))$,
\begin{align}\label{eq:linsys}
\begin{split}
 \dot{x}(t) &= A x(t) + B u(t), \\
       y(t) &= C x(t),
\end{split}
\end{align}
the associated cross Gramian matrix \cite{fernando83a} is defined as the composition of the system's controllability and observability operators:

\begin{figure}[t] \centering

 \begin{subfigure}[t]{.48\textwidth}

\begin{tikzpicture}
  \begin{loglogaxis}
    [max space between ticks=70cm,
     line width=1pt,
     major tick length=2mm,
     every axis plot/.style={line width=1pt},
     every tick/.append style={line width=1pt},
     mark size=2pt,
     x dir=reverse,
     xtick={0.01,0.0001,0.000001,0.00000001,0.0000000001},
     xlabel={Prescribed Mean Proj. Error},
     ylabel={Model Reduction Error}]
    \addplot+ [mark options={}] table[x index=0, y index=1] {numex2b.dat}; 
    \addplot+ [mark options={}] table[x index=0, y index=2] {numex2b.dat}; 
    \legend{POD, HAPOD}
  \end{loglogaxis}
\end{tikzpicture}

  \caption{Actual model reduction output $\ell^2$-error of POD and distributed HAPOD for prescribed errors $\varepsilon^*$.}
  \label{fig:numex2a1}
 \end{subfigure}
~
 \begin{subfigure}[t]{.48\textwidth}

\begin{tikzpicture}
  \begin{loglogaxis}
    [legend style={at={(0.97,0.65)},anchor=east},
     max space between ticks=70cm,
     line width=1pt,
     major tick length=2mm,
     every axis plot/.style={line width=1pt},
     every tick/.append style={line width=1pt},
     mark size=2pt,
     x dir=reverse,
     xtick={0.01,0.0001,0.000001,0.00000001,0.0000000001,1e-8,1e-10},
     ytick={0.1,1.0,10.0,100.0,1000.0},
     xlabel={Prescribed Mean Proj. Error},
     ylabel={Computational Time [s]}]
    \addplot+ [mark options={}] table[x index=0, y index=1] {numex2d.dat}; 
    \addplot+ [mark options={}] table[x index=0, y index=2] {numex2d.dat};
    \addplot+ [mark options={}] table[x index=0, y index=3] {numex2d.dat}; 
    \legend{POD, HAPOD, Minimal}
  \end{loglogaxis}
\end{tikzpicture}

\caption{Computational time for POD, distributed HAPOD time (sequential computation)
        and minimal required HAPOD time if full parallelization is assumed.}
  \label{fig:numex2a2}
 \end{subfigure}

 \caption{Comparison of model reduction error and computational time for the POD and distributed HAPOD computation
 for the distributed empirical cross Gramian example (cf.\ \cref{sc:gramian}) for varying prescribed projection error.}
 \label{fig:numex2a}
\end{figure}
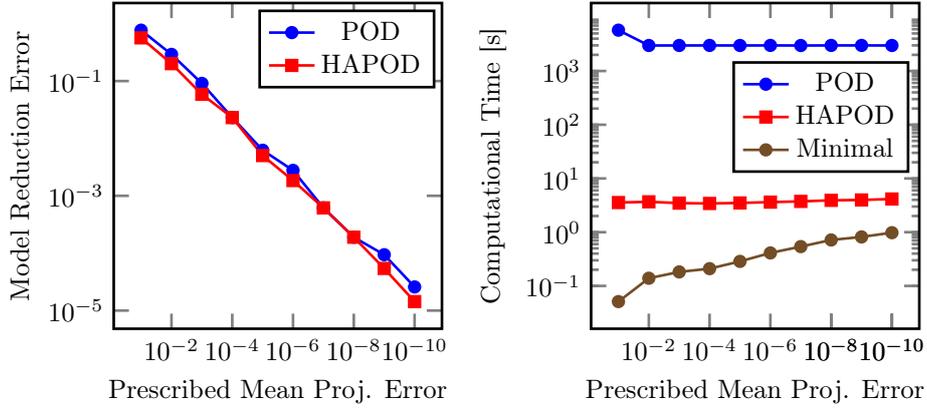

\begin{align*}
 W_X := \mathcal{CO} = \int_0^\infty \operatorname{e}^{At} BC \operatorname{e}^{At} \operatorname{d}\!t.
\end{align*}
The modes $U$ resulting from a POD of the cross Gramian constitute an approximate balancing transformation,
which can be truncated based on the associated singular values:
\begin{align*}
 W_X \stackrel{\operatorname{SVD}}{=} U D V \rightarrow U = \begin{pmatrix} U_1 & U_2 \end{pmatrix}.
\end{align*}
This truncated orthogonal projection induces a reduced order model for \eqref{eq:linsys},
\begin{align}\label{eq:gramian_reduced}
	\begin{split}
		\dot{x}_r(t) &= (U_1^\intercal A U_1) x_r(t) + (U_1^\intercal B) u(t), \\
		y_r(t) &= (C U_1) x_r(t).
	\end{split}
\end{align}
For further details we refer to \cite{sorensen02}. 
Practically, the empirical cross Gramian \cite{himpe14a} can be utilized for the computation of the cross Gramian:
\begin{align*}
     \widehat{W}_X &:= \sum_{m=1}^M \int_0^\infty \Psi^{m}(t) \operatorname{d}\!t \in \mathbb{R}^{N \times N}, \\
 \Psi^{m}_{ij}(t) &:= \langle x^m_i(t),y^j_m(t) \rangle,
\end{align*}
with $x^m(t)$ being the state trajectory for a perturbation of the $m$-th component of an impulse input and $y^j(t)$ the output trajectory for a perturbation of the $j$-th initial state component.
The empirical cross Gramian matrix may be assembled column-wise,
\begin{align}\label{eq:pwx}
\begin{split}
  \widehat{W}_X &= \Big[ \sum_{m=1}^M \int_0^\infty \psi^{m1}(t) \operatorname{d}\!t, \dots, \sum_{m=1}^M \int_0^\infty \psi^{mN}(t) \operatorname{d}\!t \Big], \\
 \psi^{mn}_i(t) &:= \langle x^m_i(t), y^n_m(t) \rangle,
\end{split}
\end{align}
by sorting the $\Psi^m(t)$ into columns.
This \emph{distributed empirical cross Gramian} together with the distributed HAPOD computation then allows a fully parallel assembly of the cross-Gramian-based approximate balancing truncated projection $U_1$.

This experiment utilizes the procedural ``Synthetic'' benchmark model\footnote{See: \url{http://modelreduction.org/index.php/Synthetic_parametric_model}} from \cite{morwiki}.
For $N = 10000$ a single-input-single-output system is generated, and we fix the parametrization to $\theta \equiv \frac{1}{10}$.
The system is excited by an impulse input $u(t) = \delta(t)$ and evolves over a time span of $T = [0,1]$ with a fixed time step width of $h = \frac{1}{100}$.
An empirical cross Gramian $\widehat{W}_X$ is computed\footnote{Computation on Intel Core i7-6700 (x86-64) CPU with $8$GB RAM.}
using \texttt{emgr}~--~empirical Gramian framework \cite{himpe13a,morHim16,emgr},
for which a regular POD and a distributed HAPOD is used to determine the left singular vectors.
For the latter, the empirical cross Gramian $\widehat{W}_X \in \R^{10000 \times 10000}$ is partitioned column-wise into $100$ blocks of size $10000 \times 100$,
which are assigned to the leafs of the distributed HAPOD tree,
and the local error tolerances chosen according to \cref{cor:hapod_bound} with $\omega = 0.5$.

\cref{fig:numex2a1} shows the error for the empirical cross Gramian-based state-space reduction comparing the original system's output and the reduced order model's output utilizing either the POD or the distributed variant of the HAPOD.
For a varying prescribed projection error, the model reduction error resulting from the POD and HAPOD,
i.e.\ the time-domain misfit between original system output and reduced-order system output measured in the $\ell^2$-norm $\varepsilon_y = \|y-y_r\|_{\ell^2}$,
decays with a similar rate as, and never exceeds the error resulting from the classic POD.

Comparing the time consumption of the POD and HAPOD, the former, due to its constant complexity, requires a fixed amount of time for each prescribed error.
The HAPOD assembly time is about three orders of magnitude smaller than for the POD and increases slowly for more accurate approximations, as shown in \cref{fig:numex2a2}.
Furthermore, if enough processor cores would be available for a full parallelization, meaning all leaf sub-PODs could be evaluated concurrently,
then for $\varepsilon^* \geq 10^{-6}$ the time requirements can be reduced again by up to one order of magnitude compared
with the single worker setup used in the experiment.
For smaller prescribed errors, the final POD starts to require a large part of the computational effort such that a balanced
tree $\T$ with depth $L_\T = 3$ would be required to gain an additional speedup.

\begin{figure}[t] \centering

 \begin{subfigure}[t]{.48\textwidth}

\begin{tikzpicture}
  \begin{loglogaxis}
    [max space between ticks=70cm,
     line width=1pt,
     major tick length=2mm,
     every axis plot/.style={line width=1pt},
     every tick/.append style={line width=1pt},
     mark size=2pt,
     ymin=1e1,
     ymax=1e5,
     ytick={1e1,1e2,1e3,1e4,1e5},
     xlabel={Block Size},
     ylabel={Speedup}]
    \addplot+ [mark options={}] table[x index=0, y expr={2980 / \thisrowno{1}}] {numex2c.dat}; 
    \addplot+ [mark options={}] table[x index=0, y expr={2980 / \thisrowno{2}}] {numex2c.dat}; 
    \addplot+ [mark options={}] table[x index=0, y expr={2980 / \thisrowno{3}}] {numex2c.dat}; 
    \addplot+ [mark options={}] table[x index=0, y expr={2980 / \thisrowno{4}}] {numex2c.dat}; 
    \legend{$L_\T = 2$, $L_\T = 3$, $L_\T = 4$, $L_\T = 5$}
  \end{loglogaxis}
\end{tikzpicture}

  \caption{Sequential runtime of the HAPOD.}
  \label{fig:numex2c1}
 \end{subfigure}
~
 \begin{subfigure}[t]{.48\textwidth}

\begin{tikzpicture}
  \begin{loglogaxis}
    [legend style={at={(0.83,0.25)},anchor=east},
     max space between ticks=70cm,
     line width=1pt,
     major tick length=2mm,
     every axis plot/.style={line width=1pt},
     every tick/.append style={line width=1pt},
     mark size=2pt,
     ymin=1e1,
     ymax=1e5,
     ytick={1e1,1e2,1e3,1e4,1e5},
     xlabel={Block Size},
     ylabel={Speedup}]
      \addplot+ [mark options={}] table[x index=0, y expr={2980 / \thisrowno{5}}] {numex2c.dat}; 
      \addplot+ [mark options={}] table[x index=0, y expr={2980 / \thisrowno{6}}] {numex2c.dat}; 
      \addplot+ [mark options={}] table[x index=0, y expr={2980 / \thisrowno{7}}] {numex2c.dat}; 
      \addplot+ [mark options={}] table[x index=0, y expr={2980 / \thisrowno{8}}] {numex2c.dat}; 
    \legend{$L_\T = 2$, $L_\T = 3$, $L_\T = 4$, $L_\T = 5$}
  \end{loglogaxis}
\end{tikzpicture}

\caption{Maximal speedup of the HAPOD assuming full parallelization.}
  \label{fig:numex2c2}
 \end{subfigure}

 \caption{Speedup of the HAPOD for balanced trees of different depth and block sizes (cf.\ \cref{sc:gramian}),
 $\varepsilon^* = 10^{-6}$, in
 comparison to the classic POD. The runtime for the classic POD is $2.98 \cdot 10^3$ seconds.}
 \label{fig:numex2c}
\end{figure}
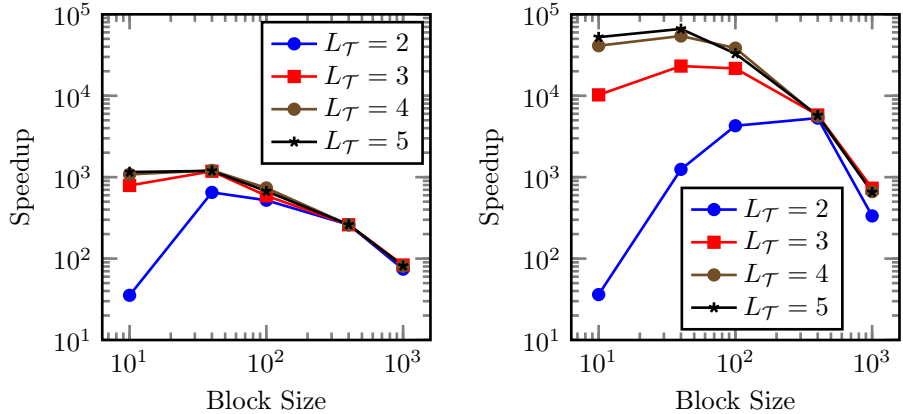

The next experiment tests the influence of the depth of the tree and the block size at the leafs on the runtime.
To this end the $10^4$ columns of the empirical cross Gramian are organized in partitions of
$10 \times 1000$, $40 \times 250$, $100 \times 100$, $400 \times 25$ and $1000 \times 10$ columns.
These partitions are each mapped to the leafs of balanced $n$-ary trees of depth $L_\T \in \{2,3,4,5\}$.
The number of children per node $n$ is determined for each tree by the number of blocks $s$ and the depth $L_\T$ of the tree via $n = \lceil s^{1/L_\T} \rceil$.

\cref{fig:numex2c} depicts the speedup of the HAPOD over a classic POD for varying tree depths and block size at the leafs.
Specifically, \cref{fig:numex2c1} shows the speedup for a sequential execution of the HAPOD, while \cref{fig:numex2c2} shows the maximal speedup assuming $s$ processors
by summing the maximum sub-POD runtimes for each level, as these sub-PODs could be processed in parallel.

This test shows that (balanced) trees with smaller blocks are preferable in terms of runtime (\cref{fig:numex2c1}).
For highly parallel computations, trees with small block sizes and more levels (depth) perform better (\cref{fig:numex2c2}).
While the two-level tree with smallest block size performs worst in comparison,
the larger the individual leaf block, the more similar are the runtimes independent from tree depth.

\subsection{Reduction of a Large Kinetic Equation Model}\label{sc:boltzmann}

\begin{figure} \centering

 \begin{subfigure}[t]{.22\textwidth}
\includegraphics[width=\textwidth]{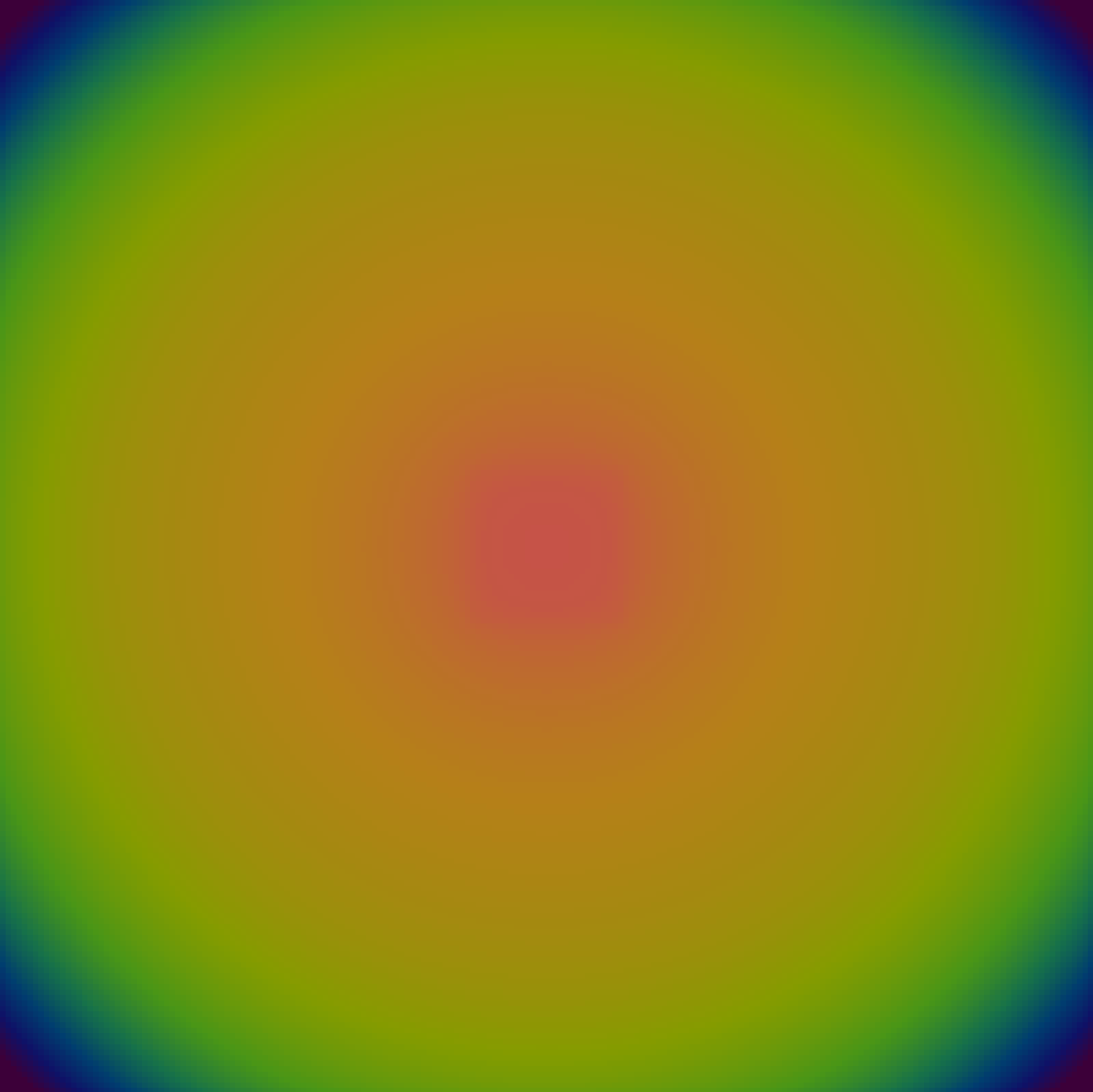}
\caption{$\mu = (0,0,0)$}
\label{fig:boltzmannsolutions1}
 \end{subfigure}
 \hspace{-0.66em}%
 ~
 \begin{subfigure}[t]{.22\textwidth}
\includegraphics[width=\textwidth]{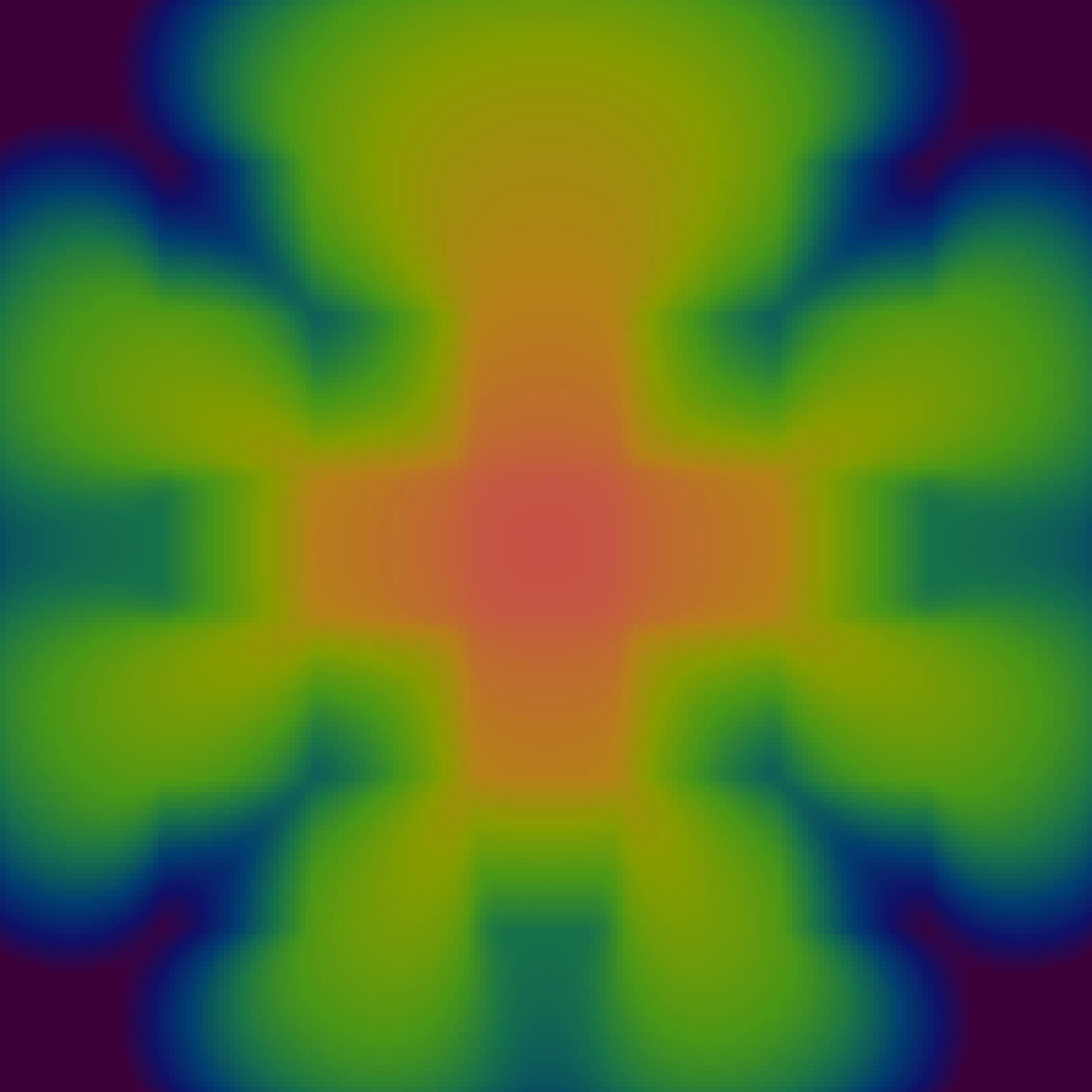}
\caption{$\mu = (0,0,6)$}
\label{fig:boltzmannsolutions2}
 \end{subfigure}
 \hspace{-0.66em}%
 ~
  \begin{subfigure}[t]{.22\textwidth}
\includegraphics[width=\textwidth]{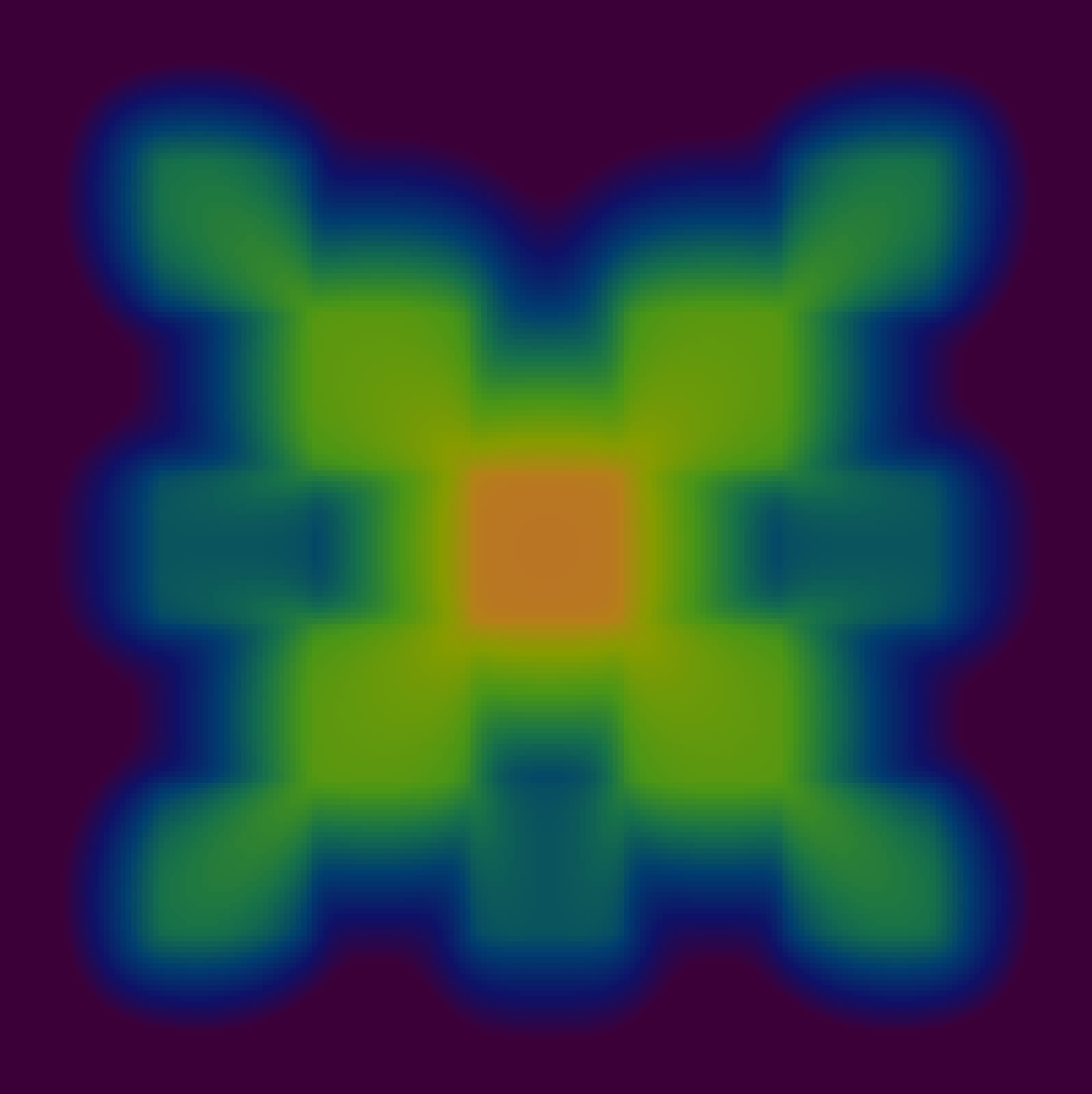}
\caption{$\mu=(2,6,0)$}
\label{fig:boltzmannsolutions3}
 \end{subfigure}
 \hspace{-0.66em}%
  ~
  \begin{subfigure}[t]{.22\textwidth}
\includegraphics[width=\textwidth]{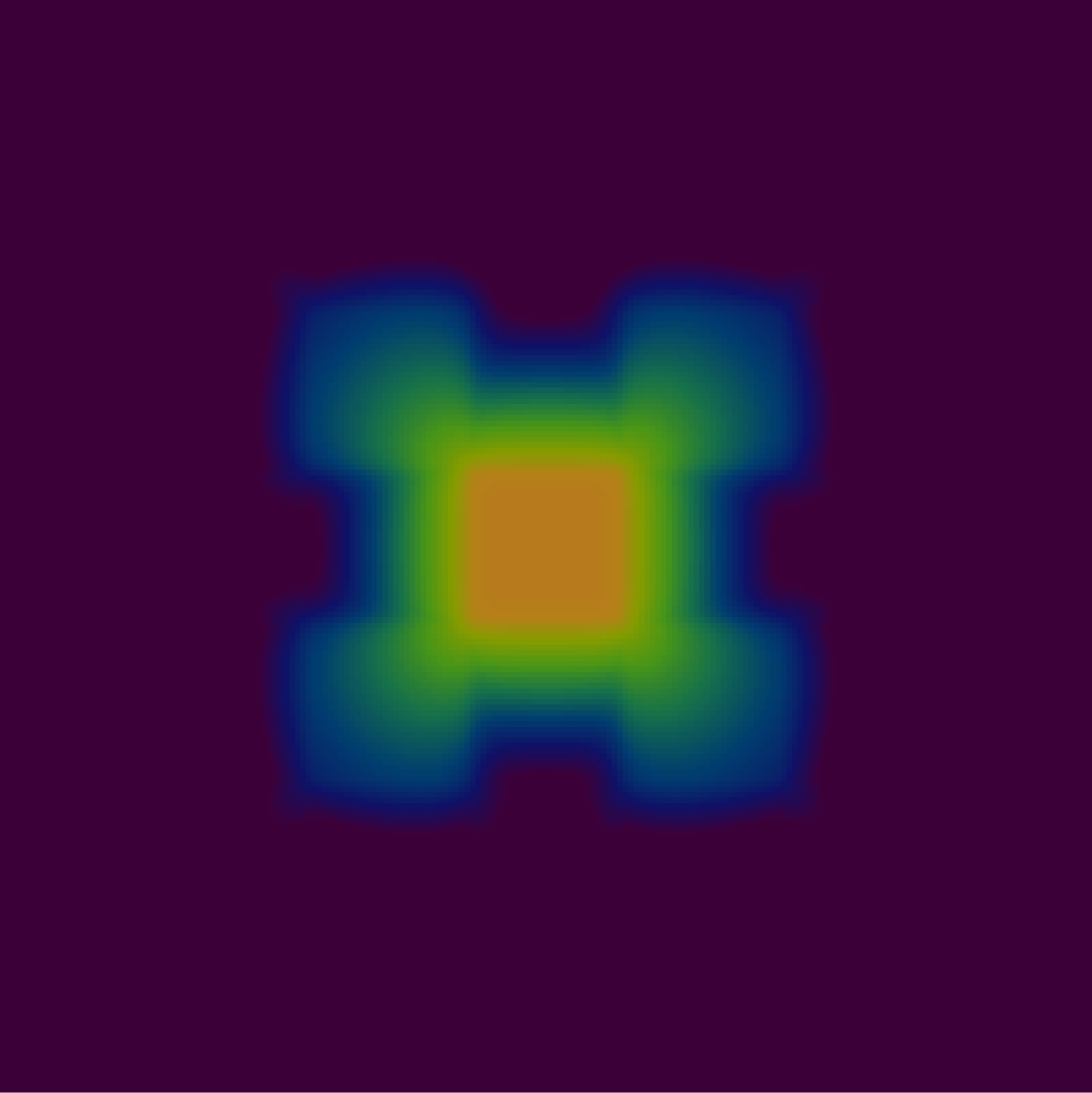}
\caption{$\mu=(8,8,4)$}
\label{fig:boltzmannsolutions4}
\end{subfigure}
 \hspace{-0.66em}%
  ~
\begin{subfigure}[t]{0.0531\textwidth}
\includegraphics[width=\textwidth]{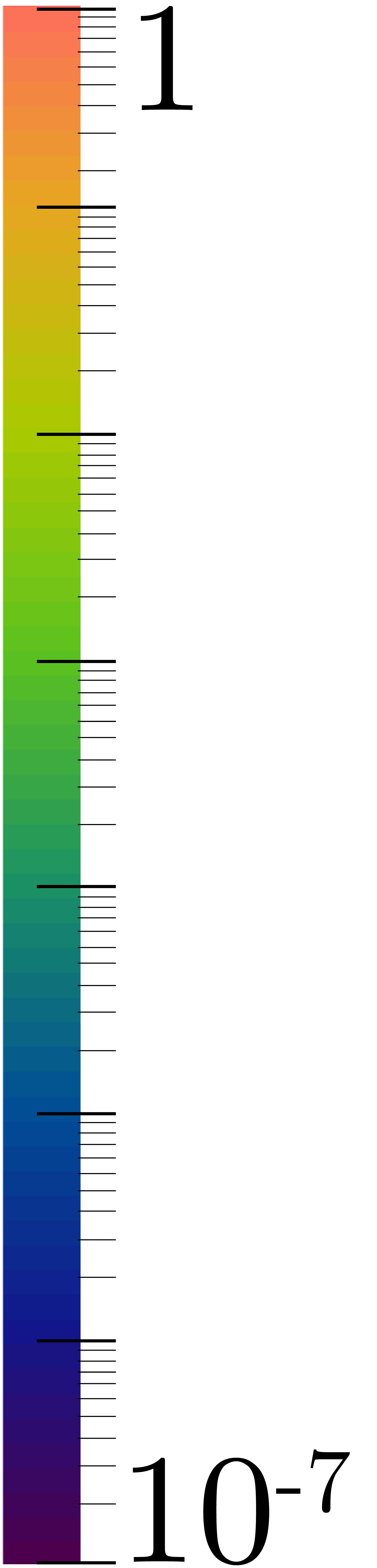}
\label{fig:boltzmannsolutions5}
 \end{subfigure}

 \caption{Solutions to the Checkerboard test case for the kinetic Boltzmann equation (cf.\ \cref{sc:boltzmann}) for different parameters
 $\mu = (\Sigma_{s,1}, \Sigma_{a,1}, \Sigma_{a,2})$. Visualized is the first component of the solution at time $T = 3.2$. The color scale is logarithmic.}
 \label{fig:boltzmannsolutions}
\end{figure}

The third numerical experiment utilizes a kinetic equation model.
In such models, the solution field does not only depend on time and space but also on velocity variables.
Hence, directly solving a kinetic equation with standard numerical methods often causes a prohibitive amount of computational cost
due to the curse of dimensionality.
Moment closure models are one approach to overcome this difficulty by transferring the kinetic equation to a hyperbolic system of coupled equations which do not depend on the velocity variable anymore (see \cite{alldredge12, brunner05, schneider14} and references therein).
This significantly reduces the effort needed to solve the problem, especially in several space dimensions.
However, the computational cost may still be too high to solve a parameter-dependent problem for a large set of parameters in a reasonable amount of time.
In this case, a POD-based state-space Galerkin projection similar to \cref{eq:gramian_reduced} can be used to further reduce the model.

Our experiment is based on the checkerboard test case for the $P_{15}$ moment closure approximation of the Boltzmann equation for neutron transport from \cite{brunner05}.
The model equation in two dimensions is given by:
\begin{equation*}
\partial_t \mathbf{p}(t,\mathbf{x}) + \mathbf{A}_x \partial_x \mathbf{p}(t,\mathbf{x}) + \mathbf{A}_z \partial_z \mathbf{p}(t,\mathbf{x}) = \mathbf{s}(t,\mathbf{x}) + \left(\Sigma_s(\mathbf{x})\mathbf{Q} - \Sigma_t(\mathbf{x}) \mathbf{I}\right) \mathbf{p}(t,\mathbf{x}),
\end{equation*}
where $\mathbf{p}(t, \mathbf{x}) \in \mathbb{R}^{136}$ for fixed spatial coordinates $\mathbf{x} = (x,z)$ and time $t$, $\mathbf{I}$ is the identity matrix and $\mathbf{Q}_{00} = 1$, $\mathbf{Q}_{ij} = 0$ otherwise.
The positive coefficients $\Sigma_s$ and $\Sigma_t = \Sigma_s + \Sigma_a$ describe scattering and total cross section, respectively, and $\mathbf{s}$ is a particle source.
The matrices $\mathbf{A}_x$, $\mathbf{A}_z \in \mathbb{R}^{136\times136}$ which describe the coupling between the moments are sparse with at most four and two entries per row, respectively.
See \cite[Eq.~8,~9]{brunner05} for the detailed definitions of the matrices.

The test case assumes a spatial domain $[0,7]\times[0,7]$ that is divided in $49$ axis-parallel cubes with unit edge width and composed of two different materials (see \cref{fig:checkerboard}) that are characterized by their scattering and absorption cross-section $\Sigma_s$ and $\Sigma_a$, respectively.
Initially, there are no neutrons in the domain.
At time $t=0$, a neutron source $\mathbf{s} = (1, 0, \ldots, 0)^\intercal$ is turned on in the center region.

The parameter dependence for the scattering and absorption cross-sections $\Sigma_{s,1}$ and $\Sigma_{a,1}$ for the first material (red regions in \cref{fig:checkerboard}) and the absorption cross-section $\Sigma_{a,2}$ for the second material (black regions in \cref{fig:checkerboard}) is to be retained for the reduced order model,
while the scattering cross-section of the second material is fixed to $\Sigma_{s,2} = 0$.
The three parameters $\Sigma_{s,1}$, $\Sigma_{a,1}$, $\Sigma_{a,2}$ are each chosen in the range $[0,8]$.
For the POD, each parameter is uniformly sampled by the five values $\{0, 2, 4, 6, 8\}$ such that $125$ solution trajectories have to be calculated. 

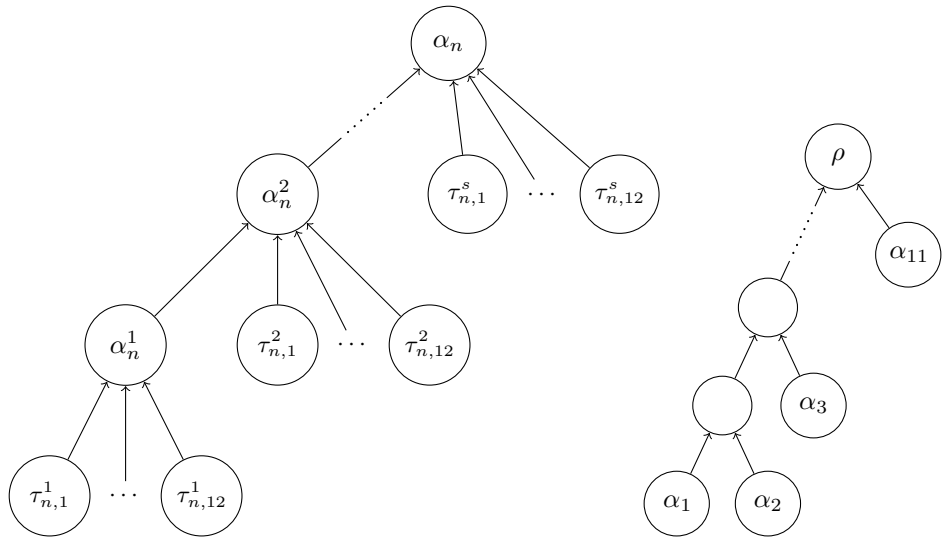
\begin{figure}\centering
 \tikzset{
	->-/.style={
 	decoration={markings,
	            mark=at position 0.5 with {\draw [white, ultra thick, -] (-0.4,0) -- (0.4,0);
			               	       \draw [line cap=round, dash pattern=on 0pt off 4\pgflinewidth, thick, -] (-0.3,0) -- (0.3,0);}},
	postaction={decorate}
        }
 }

 \begin{subfigure}[t]{.66\textwidth}
 \begin{tikzpicture}
  [level distance=20mm,
   thin,
   every node/.style={text width=0.65cm, align=center},
   level 1/.style={sibling distance=5mm},
   level 2/.style={sibling distance=5mm},
   level 3/.style={sibling distance=5mm},
   edge from parent/.style={draw,>-,>=to reversed}
  ]
  \node[draw, circle] {$\alpha_n$}
          child { node [circle, draw] {$\alpha_n^2$} edge from parent[->-,draw,>=to reversed]
		  child { node [circle, draw] {$\alpha_n^1$}    
			  child {node [circle, draw] {\small $\tau_{n, 1}^1$}}
			  child {node {\small \phantom{$\cdots$}} edge from parent[draw=none]}
                          child {node {\small $\cdots$}}
			  child {node {\small \phantom{$\cdots$}} edge from parent[draw=none]}
			  child {node [circle, draw] {\small $\tau_{n, 12}^{1}$}}}
	          child {node {\small \phantom{$\cdots$}} edge from parent[draw=none] }
	          child {node {\small \phantom{$\cdots$}} edge from parent[draw=none] }
	          child {node {\small \phantom{$\cdots$}} edge from parent[draw=none] }
	          child {node [circle, draw] {\small $\tau_{n, 1}^2$}}
	          child {node {\small \phantom{$\cdots$}} edge from parent[draw=none]}
                  child {node {\small $\cdots$}}
		  child {node {\small \phantom{$\cdots$}} edge from parent[draw=none]}
	          child {node [circle, draw] {\small $\tau_{n, 12}^2$}}}
          child {node {\small \phantom{$\cdots$}} edge from parent[draw=none] }
          child {node {\small \phantom{$\cdots$}} edge from parent[draw=none] }
          child {node {\small \phantom{$\cdots$}} edge from parent[draw=none] }
          child {node {\small \phantom{$\cdots$}} edge from parent[draw=none] }
	  child {node [circle, draw] {\small $\tau_{n, 1}^{s}$}}
          child {node {\small \phantom{$\cdots$}} edge from parent[draw=none]}
	  child {node {\small $\cdots$}}
          child {node {\small \phantom{$\cdots$}} edge from parent[draw=none]}
	  child {node [circle, draw] {\small $\tau_{n, 12}^{s}$}}
    ;
\end{tikzpicture}\caption{HAPOD on compute node $n$. The time steps are split into $s$ slices ($s = \lceil (2n_t + 1)/l \rceil$).
                          Concurrently, each of the 12 processor cores calculates one chunk at a time, performs a POD and sends the
		          resulting modes to the main MPI rank on the processor.
		          $\tau_{n,c}^t$: $t$-th time slice on core $c$.}%
\label{subfig:HAPODonNode}
 \end{subfigure}
 ~
  \begin{subfigure}[t]{.30\textwidth}
\begin{tikzpicture}
  [thin,
   every node/.style={text width=0.51cm, align=center},
   level 1/.style={sibling distance=18.462mm, level distance=20mm,},
   level 2/.style={sibling distance=12mm, level distance=13mm},
   edge from parent/.style={draw,>-,>=to reversed},
  ]
  \node [draw, circle] {$\rho$}
	    child {node [draw, circle] {} edge from parent[->-,draw,>=to reversed]
		   child  {node [draw, circle] {} 
		           child  {node [draw, circle] {$\alpha_1$}}
	                   child  {node [draw, circle] {$\alpha_2$}}}
		   child  {node [draw, circle] {$\alpha_3$}}}
	    child {node [draw, circle, yshift=7mm]{$\alpha_{11}$}};
    \end{tikzpicture}\caption{An incremental \mbox{HAPOD} (cf.\ \cref{sc:special_cases}) is performed on MPI rank 0 with the modes collected on each node.
                          $\alpha_n$: modes from node $n$.}
\label{subfig:HAPODoverNodes}
\end{subfigure}
\caption{HAPOD tree used for kinetic Boltzmann example (cf.\ \cref{sc:boltzmann}) on 11 compute nodes with 12 cores each.}
 \label{fig:boltzmannhapod}
\end{figure}

The model is solved by a finite volume solver for systems of hyperbolic equations implemented in dune-gdt \cite{leibner15, dunegdt},
using a numerical Lax-Friedrichs flux and an explicit Euler fractional step time stepping scheme (see \cite[Ch.~17.1]{leveque02}) to incorporate the right-hand side into the solution.
Solutions for some exemplary parameter choices are visualized in \cref{fig:boltzmannsolutions}.

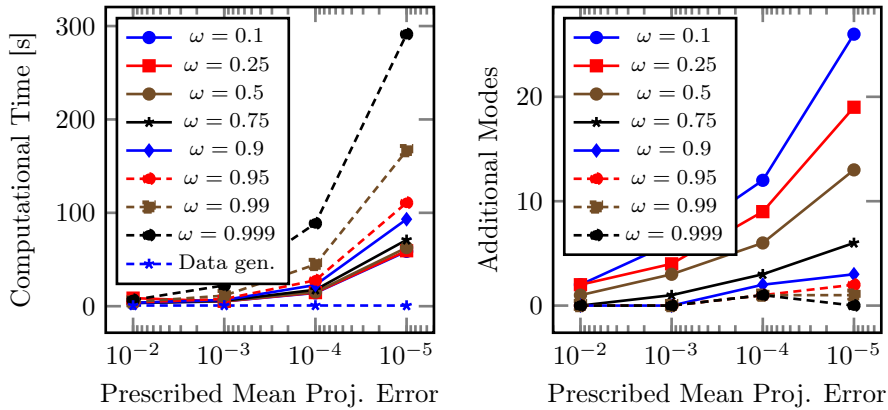
\begin{figure} \centering

 \begin{subfigure}[t]{.48\textwidth}

\begin{tikzpicture}
  \begin{semilogxaxis}
    [legend style={legend pos=north west,font=\footnotesize},
     max space between ticks=70cm,
     line width=1pt,
     major tick length=2mm,
     every axis plot/.style={line width=1pt},
     every tick/.append style={line width=1pt},
     mark size=2pt,
     x dir=reverse,
     xlabel={Prescribed Mean Proj. Error},
     ylabel={Computational Time [s]}]
    \addplot+ [mark options={}] table[x index=0, y index=2] {time_gridsize_20_initialtol_0.01_chunk_10_omega_0.1.dat};
    \addplot+ [mark options={}] table[x index=0, y index=2] {time_gridsize_20_initialtol_0.01_chunk_10_omega_0.25.dat};
    \addplot+ [mark options={}] table[x index=0, y index=2] {time_gridsize_20_initialtol_0.01_chunk_10_omega_0.5.dat};
    \addplot+ [mark options={}] table[x index=0, y index=2] {time_gridsize_20_initialtol_0.01_chunk_10_omega_0.75.dat};
    \addplot+ [mark options={}] table[x index=0, y index=2] {time_gridsize_20_initialtol_0.01_chunk_10_omega_0.9.dat};
    \addplot+ [mark options={}] table[x index=0, y index=2] {time_gridsize_20_initialtol_0.01_chunk_10_omega_0.95.dat};
    \addplot+ [mark options={}] table[x index=0, y index=2] {time_gridsize_20_initialtol_0.01_chunk_10_omega_0.99.dat};
    \addplot+ [mark options={}] table[x index=0, y index=2] {time_gridsize_20_initialtol_0.01_chunk_10_omega_0.999.dat};
    \addplot+ [mark options={}] table[x index=0, y index=3] {time_gridsize_20_initialtol_0.01_chunk_10_omega_0.999.dat};
    \legend{$\omega = 0.1$, $\omega = 0.25$, $\omega = 0.5$, $\omega = 0.75$, $\omega = 0.9$, $\omega = 0.95$, $\omega = 0.99$, $\omega = 0.999$, Data gen.}
  \end{semilogxaxis}
\end{tikzpicture}

  \caption{HAPOD execution wall time for different values of $\omega$.
  For all values of $\omega$, the HAPOD is much faster than the POD which took about 1600 seconds for each prescribed tolerance $\varepsilon^*$.
  Snapshot generation (Data gen.) took 0.8 seconds.}
  \label{subfig:diffomega2}
 \end{subfigure}
 ~
  \begin{subfigure}[t]{.48\textwidth}

\begin{tikzpicture}
  \begin{semilogxaxis}
    [legend style={legend pos=north west,font=\footnotesize},
     max space between ticks=70cm,
     line width=1pt,
     major tick length=2mm,
     every axis plot/.style={line width=1pt},
     every tick/.append style={line width=1pt},
     mark size=2pt,
     x dir=reverse,
     xlabel={Prescribed Mean Proj. Error},
     ylabel={Additional Modes}]
    \addplot+ [mark options={}] table[x index=0, y index=4] {num_modes_gridsize_20_initialtol_0.01_chunk_10_omega_0.1.dat};
    \addplot+ [mark options={}] table[x index=0, y index=4] {num_modes_gridsize_20_initialtol_0.01_chunk_10_omega_0.25.dat};
    \addplot+ [mark options={}] table[x index=0, y index=4] {num_modes_gridsize_20_initialtol_0.01_chunk_10_omega_0.5.dat};
    \addplot+ [mark options={}] table[x index=0, y index=4] {num_modes_gridsize_20_initialtol_0.01_chunk_10_omega_0.75.dat};
    \addplot+ [mark options={}] table[x index=0, y index=4] {num_modes_gridsize_20_initialtol_0.01_chunk_10_omega_0.9.dat};
    \addplot+ [mark options={}] table[x index=0, y index=4] {num_modes_gridsize_20_initialtol_0.01_chunk_10_omega_0.95.dat};
    \addplot+ [mark options={}] table[x index=0, y index=4] {num_modes_gridsize_20_initialtol_0.01_chunk_10_omega_0.99.dat};
    \addplot+ [mark options={}] table[x index=0, y index=4] {num_modes_gridsize_20_initialtol_0.01_chunk_10_omega_0.999.dat};
    \legend{$\omega = 0.1$, $\omega = 0.25$, $\omega = 0.5$, $\omega = 0.75$, $\omega = 0.9$, $\omega = 0.95$, $\omega = 0.99$, $\omega = 0.999$}
  \end{semilogxaxis}
\end{tikzpicture}

  \caption{Number of additional HAPOD modes (compared to POD) for different values of $\omega$.
  The POD resulted in $2$, $10$, $35$ and $94$ modes for a prescribed error $\varepsilon^*$ of $10^{-2}$, $10^{-3}$, $10^{-4}$ and $10^{-5}$, respectively.}
  \label{subfig:diffomega1}
 \end{subfigure}

 \caption{Influence of $\omega$ on HAPOD execution wall time and number of resulting modes for the kinetic Boltzmann equation example (cf.\ \cref{sc:boltzmann}) on a grid with $k^2 = 400$ elements ($N=54400$ degrees of freedom).}
 \label{fig:diffomega}
\end{figure}

As the $P_{15}$ model consists of $136$ coupled equations with $136$ unknowns and the finite volume scheme uses a uniform cube grid with $k^2$ elements,
the discrete solution vector for the finite volume discretization at a fixed time contains $N = 136 k^2$ entries.
The test case is solved up to a time of $T = 3.2$ and the time step length is determined by a Courant–Friedrichs–Lewy number of $0.4$ which leads to  $n_t = \left\lceil \frac{T}{7/k\cdot 0.4} \right\rceil$ time steps per trajectory.
To obtain an accurate reduced order model, the intermediate steps in the fractional step discretization have to be included into the snapshot set as well, such that $2n_t$ discrete solution vectors have to be stored per trajectory.
Thus, a total of approximately $250 n_t$ snapshots has to be handled.
This corresponds to roughly $250 \cdot \frac{T}{7/k\cdot 0.4} \cdot 136 k^2 \approx 39000 k^3$ double precision floating point numbers that have to be stored in memory.
For a grid with $k=40$, these would take about $20$ gigabytes of memory whereas for $k=200$ about $2.5$ terabytes of memory were needed. 

\begin{figure}[t] \centering

 \begin{subfigure}[t]{.48\textwidth}
 \begin{tikzpicture}
  \begin{semilogxaxis}
    [legend style={legend pos=north west,font=\footnotesize},
     max space between ticks=70cm,
     line width=1pt,
     major tick length=2mm,
     every axis plot/.style={line width=1pt},
     every tick/.append style={line width=1pt},
     mark size=2pt,
     x dir=reverse,
     xlabel={Prescribed Mean Proj. Error},
     ylabel={Max.\ Intermed.\ Modes}]
    \addplot+ [mark options={}] table[x index=0, y index=3] {num_modes_gridsize_20_initialtol_0.01_chunk_10_omega_0.1.dat};
    \addplot+ [mark options={}] table[x index=0, y index=3] {num_modes_gridsize_20_initialtol_0.01_chunk_10_omega_0.25.dat};
    \addplot+ [mark options={}] table[x index=0, y index=3] {num_modes_gridsize_20_initialtol_0.01_chunk_10_omega_0.5.dat};
    \addplot+ [mark options={}] table[x index=0, y index=3] {num_modes_gridsize_20_initialtol_0.01_chunk_10_omega_0.75.dat};
    \addplot+ [mark options={}] table[x index=0, y index=3] {num_modes_gridsize_20_initialtol_0.01_chunk_10_omega_0.9.dat};
    \addplot+ [mark options={}] table[x index=0, y index=3] {num_modes_gridsize_20_initialtol_0.01_chunk_10_omega_0.95.dat};
    \addplot+ [mark options={}] table[x index=0, y index=3] {num_modes_gridsize_20_initialtol_0.01_chunk_10_omega_0.99.dat};
    \addplot+ [mark options={}] table[x index=0, y index=3] {num_modes_gridsize_20_initialtol_0.01_chunk_10_omega_0.999.dat};
    \legend{$\omega = 0.1$, $\omega = 0.25$, $\omega = 0.5$, $\omega = 0.75$, $\omega = 0.9$, $\omega = 0.95$, $\omega = 0.99$, $\omega = 0.999$}
  \end{semilogxaxis}
\end{tikzpicture}
  \caption{Maximal number of intermediate \mbox{HAPOD} modes for different values of $\omega$.}
  \label{subfig:diffomega3}
 \end{subfigure}
 ~
  \begin{subfigure}[t]{.48\textwidth}
  
  \begin{tikzpicture}
   \begin{loglogaxis}
     [legend style={legend pos=north east,font=\footnotesize},
      max space between ticks=70cm,
      line width=1pt,
      major tick length=2mm,
      every axis plot/.style={line width=1pt},
      every tick/.append style={line width=1pt},
      mark size=2pt,
      x dir=reverse,
      xlabel={Prescribed Mean Proj. Error},
      ylabel={Mean Model Reduction Error}]
     \addplot+ [mark options={}] table[x index=0, y index=9] {l2_mean_errs_gridsize_20_chunk_10_omega_0.1.dat};
     \addplot+ [mark options={}] table[x index=0, y index=9] {l2_mean_errs_gridsize_20_chunk_10_omega_0.25.dat};
     \addplot+ [mark options={}] table[x index=0, y index=9] {l2_mean_errs_gridsize_20_chunk_10_omega_0.5.dat};
     \pgfplotsset{cycle list shift=1}
     \addplot+ [mark options={}] table[x index=0, y index=9] {l2_mean_errs_gridsize_20_chunk_10_omega_0.9.dat};
     \pgfplotsset{cycle list shift=3}
     \addplot+ [mark options={}] table[x index=0, y index=9] {l2_mean_errs_gridsize_20_chunk_10_omega_0.999.dat};
     \addplot+ [mark options={}, mark=x] table[x index=0, y index=3] {l2_mean_errs_gridsize_20_chunk_10_omega_0.95.dat};
     \legend{$\omega = 0.1$, $\omega = 0.25$, $\omega = 0.5$, $\omega = 0.9$, $\omega = 0.999$, POD}
   \end{loglogaxis}
  \end{tikzpicture}

   \caption{$\ell^2$-mean model reduction errors for 1250 random parameters, $k = 20$.}
   \label{subfig:numex3c}
  \end{subfigure}

 \caption{Number of local HAPOD modes and model reduction errors for the kinetic Boltzmann equation example (cf.\ \cref{sc:boltzmann}).}
 \label{fig:numex3s2}
\end{figure}
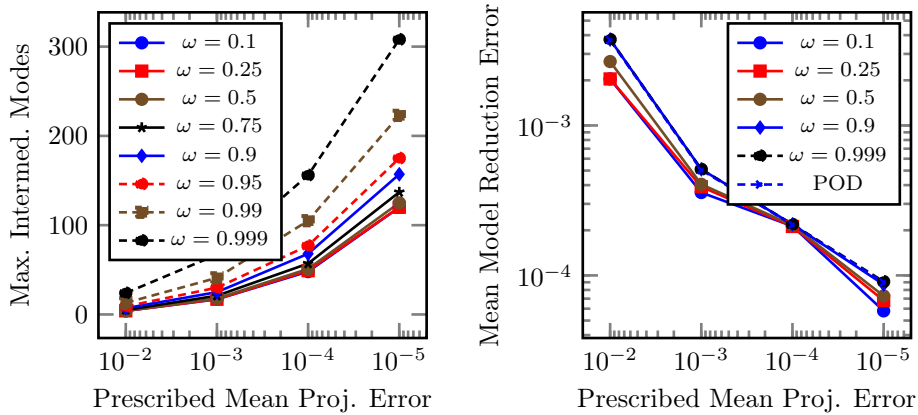

The numerical experiments are performed on eleven compute nodes of a distributed memory computer cluster\footnote{Each node encloses two Intel Xeon Westmere X5650 CPUs ($2 \times 6$ cores) and $48$GB RAM.} utilizing $125$ processor cores.
In the case of the classical POD, each processor core calculates a solution trajectory for one parameter of the sample parameter set, after which the resulting discrete solution vectors are gathered on a single node where the POD is performed.
For the HAPOD, the local PODs are calculated in parallel whenever possible. On each core a chunk of $l = 10$ time steps is calculated at a time,
a POD is performed with this chunk per core and the remaining modes are gathered per node and another POD is computed.
Subsequently, the next solution chunk is calculated and compressed by a POD on each core. The resulting modes together with the modes from the first POD on node level serve as input to a second POD on node level.
This is repeated until all time steps are calculated (cf.\ \cref{subfig:HAPODonNode}).
The result is a set of modes on each node.
Instead of gathering all modes on the main node at once, which would exceed the main node's memory, the modes are sequentially sent to the main node where additional PODs for each node are performed (cf.\ \cref{subfig:HAPODoverNodes}).
The underlying POD algorithm is provided by \texttt{pyMOR} \cite{milk-rave-schindler2016,pymor}, which is also used to compute and solve the
resulting reduced order model. We use an optimized, incremental variant of the POD algorithm in \cref{def:method_of_snapshots}, which exploits
the block structure of the Gramian with the diagonal blocks being given by diagonal matrices containing the singular values
of the PODs performed at the child nodes.
For $k=60$, $\omega=0.95$ and $\epsilon^\ast = 10^{-4}$, this improved the overall HAPOD computation time compared to the unoptimized algorithm by 7.4\% from 457 to 423 seconds.

In \cref{fig:diffomega}, computational time and number of HAPOD modes for different values of $\omega$ (see \cref{cor:hapod_bound}) are plotted against the prescribed $\ell^2$-mean error tolerance.
A $20\times20$ grid was used ($k = 20$, $N=54400$).
With decreasing $\omega$, the computational time for the HAPOD reduces but the number of final modes required to satisfy the error bound increases.
Thus, choosing a larger value of $\omega$ means trading some time spent in the HAPOD for a more efficient reduced model.

Computing the classical POD takes about 1600 seconds for each tolerance.
As for the previous numerical examples,
the HAPOD is notably faster than the POD for all tested tolerances (see \cref{subfig:diffomega2}).
Note that the HAPOD is about five times as fast as the POD, even for $\omega = 0.999$ where at most one additional final mode is obtained.
The snapshot generation, i.e.\ the solution of the high-dimensional problem, takes only a few seconds for this grid size, so the overall
computational time is dominated by the POD computation.

The maximal number of intermediate modes increases with $\omega$ (see \cref{subfig:diffomega3}).
This may be important in terms of memory usage, especially if the intermediate modes are gathered in one node's memory at some time during the HAPOD.
A smaller value of $\omega$ may thus be preferable if a shortage of memory is expected.
Choosing $\omega = 0.95$, the number of final HAPOD modes is only slightly higher than the number of POD modes (at most two additional modes are needed), while the computation is, depending on the tolerance, at least one order of magnitude faster.

To get a measure for the model reduction error, the reduced model was solved for 1250 random combinations of $\Sigma_{s,1}$, $\Sigma_{a,1}$, $\Sigma_{a,2} \in [0,8]$ and compared to the high-dimensional solution.
For $\omega$ close to one, the resulting $\ell^2$-mean error is almost equal for POD and HAPOD (see
\cref{subfig:numex3c}). For small values of $\omega$, the model reduction error decreases slightly due to the larger
number of HAPOD modes, which here result in slightly better approximation spaces than backed by theory. Solving the reduced model takes about $5\cdot 10^{-2}$ seconds independent of the grid size and is thus considerably faster than solving the full model which takes up to $500$ seconds on a $200\times200$ grid.

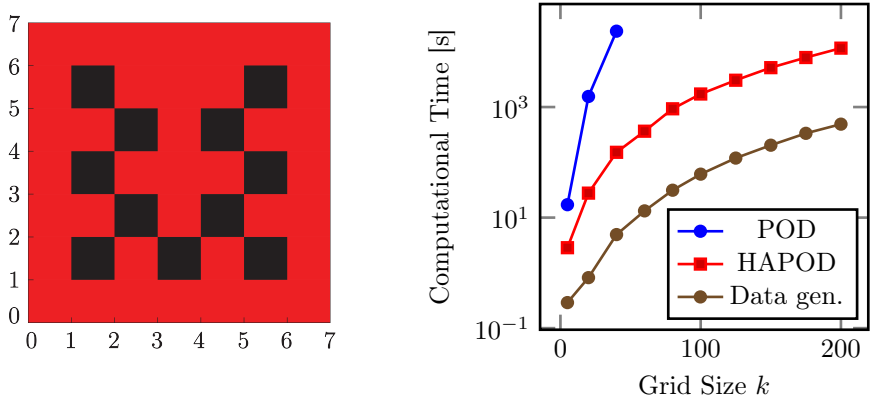
\begin{figure}[t] \centering
 \begin{subfigure}[t]{.48\textwidth}
    \begin{center}
		  \raisebox{6.5mm}{\includegraphics[width=.74\textwidth]{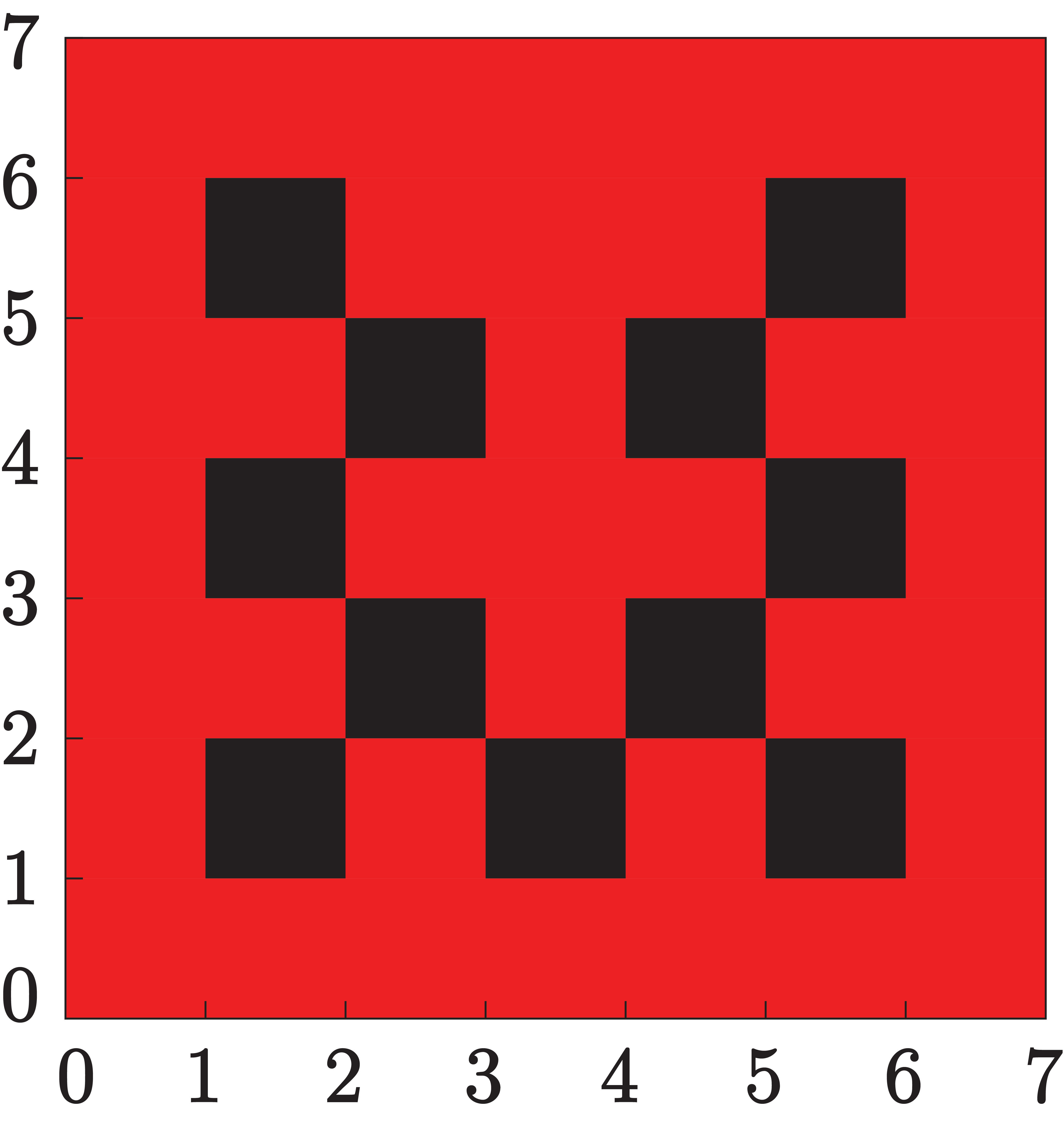}}
	\end{center}%
    \vspace*{-0.355cm}
   \caption{Computational domain: red and black regions represent common materials.}
   \label{fig:checkerboard}
 \end{subfigure}
 ~
  \begin{subfigure}[t]{.48\textwidth}
  \begin{tikzpicture}
  \begin{semilogyaxis}
    [legend pos=south east,
     max space between ticks=50cm,
     line width=1pt,
     major tick length=2mm,
     every axis plot/.style={line width=1pt},
     every tick/.append style={line width=1pt},
     mark size=2pt,
     xlabel={Grid Size $k$},
     ylabel={Computational Time [s]}]
    \addplot+ [mark options={}] table[x index=0, y index=1] {time_tol_1e-4_chunk_10_omega_0.95.dat};
    \addplot+ table[x index=0, y index=2] {time_tol_1e-4_chunk_10_omega_0.95.dat};
    \addplot+ [mark options={}] table[x index=0, y index=3] {time_tol_1e-4_chunk_10_omega_0.95.dat};
    \legend{POD, HAPOD, Data gen.}
  \end{semilogyaxis}
\end{tikzpicture}
  \caption{Computational wall time for POD and \mbox{HAPOD} ($\varepsilon^\ast = 10^{-4}$, $\omega = 0.95$).
}
  \label{subfig:modred}
 \end{subfigure}

 \caption{Computational domain and required time for the kinetic Boltzmann equation example (cf.\ \cref{sc:boltzmann}).}
\end{figure}

The previous tests were performed on a coarse $20 \times 20$ grid.
Since the memory consumption scales with $k^3$, refining the grid quickly leads to a situation where the snapshots do not fit in memory simultaneously such that a classical POD cannot be performed without access to mass storage.
In \cref{subfig:modred}, a performance comparison between POD and HAPOD ($\omega = 0.95$) for different grid sizes can be found.
The HAPOD is up to two orders of magnitude faster than the POD for the coarse grids where the POD is still feasible.
For $k \geq 60$, the POD fails to run due to memory limitations while the HAPOD does not have this problem.
Note that the HAPOD is twice as fast on the $200\times200$ grid than the classical POD on a $40 \times 40$ grid even though the amount of data that needs to be processed increases by a factor of $125$ between $k=40$ and $k=200$.
The time used for data generation plays a negligible role in the algorithm.
Creating the snapshots for POD and HAPOD takes less than 10 seconds for $k=40$ and about 500 seconds for $k=200$.
Using the HAPOD thus directly translates into a much faster overall reduced basis generation.

The final incremental PODs performed to collect the outputs of the individual compute nodes (\cref{subfig:HAPODoverNodes}) are not optimal in terms of parallelism as all calculations are done on the main node. We thus tested another tree topology where a binary tree of nodes is built. 
Indeed this improved computational wall times of the HAPOD again, e.g.\ from 423 to 239 seconds (43\% reduction) for $k=60$, $\omega=0.95$, $\epsilon^\ast = 10^{-4}$, while the memory requirements and the quality
of the resulting HAPOD space were comparable.

\section{Conclusion}\label{sc:conc}
With the HAPOD, this work introduces a general scheme for approximate POD
computation that allows to distribute the computational workload among arbitrary trees
of workers, making it easily adaptable to different computing environments.
Rigorous error and mode bounds are proven that ascertain the reliability
and performance of the method. Specialized variants for incremental and distributed
HAPOD computation are discussed, and numerical experiments underscore the
applicability of the HAPOD, from small embedded devices to
high performance computer clusters.

\section*{Code Availability}
The source code used to compute the presented results is available under open source licenses
and is included in the supplementary material to this publication.

\bibliographystyle{plain}
\bibliography{main}

\end{document}